\documentclass[numbering,hidelink]{myArt}
\usepackage{enumitem}
\usepackage{lineno}

\title[Generic regularity for the Stefan problem in 4+1 dimensions]{Generic regularity in time for solutions\\of the Stefan problem in 4+1 dimensions}
\date{}
\author{Giacomo Colombo}
\address{ETH Z\"urich, Department of Mathematics, R\"amistrasse 101, 8092, Z\"urich, Switzerland}
\email{giacomo.colombo@math.ethz.ch}
\begin{document}
\begin{abstract}
    We show that the free boundary of a solution of the Stefan problem in $\mathbb R^{4+1}$ is a $3$-dimensional manifold of class $C^\infty$ in $\R^4$ for almost every time.
    This is achieved by showing that for all dimensions $n$ the singular set $\Sigma\subset \R^{n+1}$ can be decomposed in two parts $\Sigma=\Sigma^\infty\cup \Sigma^*$, where $\Sigma^\infty$ is covered by one $(n-1)$-dimensional manifold of class $C^\infty$ in $\R^{n+1}$
    and its projection onto the time axis has Hausdorff dimension 0,
    while $\Sigma^*$ 
    is parabolically countably $(n-2)$-rectifiable.
\end{abstract}
\maketitle
\section{Introduction}
The Stefan problem~\cite{Stefan} describes the evolution of a block of ice melting in water. More precisely, the temperature $\theta\ge0$ satisfies $\partial_t\theta=\Delta\theta$ in $\{\theta>0\}$ (the region occupied by water) with the additional boundary condition $\theta_t=|\nabla_x\theta|^2$ on $\partial\{\theta>0\}$ (the interface ice/water).
Assuming non-zero speed of the free boundary $\partial\{\theta>0\}$, the Stefan problem is locally equivalent through the change of variables $u(x,t)=\int_0^t\theta(x,s)ds$ (see \cite{Duvaut}) to the parabolic obstacle problem
\begin{equation}\label{eq:SP}
    \begin{cases}
        (\Delta-\partial_t)u=\chi_{\{u>0\}} &\text{in }\Omega\times[0,T],\\
        u\ge0,\partial_tu\ge0 &\text{in }\Omega\times[0,T],\\
        \partial_tu>0   &\text{in }\{u>0\},
    \end{cases}\quad \Omega\times[0,T]\subset\R^n\times\R.
\end{equation}
The regularity of the free boundary for \eqref{eq:SP} was developed in the groundbreaking \cite{Caffarelli77}. The main result shows that the free boundary is smooth outside a closed set $\Sigma$ of singular points, which present a ``cusp-like'' behaviour, namely such that the contact set $\{u=0\}$ has zero density.

\smallskip\noindent
The center of further studies has been the singular part of the free boundary. In particular, there has been interest on the study of the size of the singular set and its structure \cite{LindgrenMonneau15,Blanchet06,FRS}.
If we denote by $\Sigma_t$ the singular set at time $t$, variants of the techniques used in the elliptic counterpart show that $\Sigma_t$ is contained in an $(n-1)$-dimensional $C^1$ manifold. Moreover, the whole singular set $\Sigma$ is contained in an $(n-1)$-dimensional manifold which is $C^1$ in space and $C^{1/2}$ in time \cite{Blanchet06,LindgrenMonneau15}.

This results are optimal in space, in the sense that for a single time $t$ the free boundary can be $(n-1)$-dimensional. However, it is still open the question of what is the size \textit{in time} of the singular set, i.e. for how many times the set $\Sigma_t$ is not empty.
In the seminal paper~\cite{FRS} the authors show the sharp bound on the parabolic dimension\footnote{The parabolic Hausdorff dimension of a set is the Hausdorff dimension in by the parabolic distance $\delta((x_1,t_1),(x_2,t_2))=\sqrt{|x_1-x_2|^2+|t_1-t_2|}$.} of the singular set $\dim_\pb(\Sigma)\le n-1$.
Moreover, if $n\le3$, the authors show that for almost every time $t$ the singular set $\Sigma_t$ is empty\footnote{More precisely, if $u$ solves \eqref{eq:SP} in $n+1$ dimensions and $\cS$ denotes the set of times such that $\Sigma_t$ is not empty, then $\dim_\cH(\cS)=0$ when $n=2$ and $\dim_\cH(\cS)\le 1/2$ when $n=3$.}.

Our main result is the following.
\begin{teo}\label{teo:generic R4}
    Let $\Omega\subset\R^4$, let $u\in L^\infty(\Omega\times(0,T))$ be a solution of the Stefan problem~\eqref{eq:SP}, and let
    \[
        \cS\coloneqq \{t\in(0,T)\,:\,\exists x\in\Omega \text{ s.t. }(x,t)\in\Sigma\}.
    \]
    Then
    \[
        \cH^1(\cS)=0.
    \]
    In particular, for almost every time $t\in(0,T)$, the free boundary is a 3-dimensional manifold of class $C^\infty$ in $\R^4$.
\end{teo}
Theorem~\ref{teo:generic R4} is based on the following result, valid in all dimensions.
\begin{teo}\label{teo:smooth covering}
    Let $\Omega\subset\R^n$ and let $u$ be a bounded solution of the Stefan problem \eqref{eq:SP} in $\Omega\times[0,T]$. Then there exists $\Sigma^\infty\subset\Sigma$ that can be covered by one $(n-1)$-dimensional manifold of class $C^\infty$ in $\R^{n+1}$ and such that
    \[
        \dim_\cH(\{t\in (0,T)\,:\, \exists (x,t)\in\Sigma^\infty\})=0,
    \]
    while $\Sigma\setminus\Sigma^\infty$ is parabolically countably $(n-2)$-rectifiable\footnote{A set is parabolically countably $m$-rectifiable if it can be covered by countably many sets of the form $f(E)$, where $E\subset\R^m$ and $f\colon E\to \R^{n+1}$ is Lipschitz with respect to the parabolic distance. By~\cite[Theorem 6.1]{Mattila}, this is stronger than rectifiability with respect to the euclidean structure.}.
\end{teo}

\begin{oss}
    The approach is solid and we expect it to apply also to less regular right hand sides.
    More precisely, assume that $u\ge0$ satisfying $\partial_tu\ge0$ and $\partial_tu>0$ in $\{u>0\}$ solves $(\Delta-\partial_t) u = f\chi_{\{u>0\}}$ with $f>0$ and $f\in C^{h,\beta}$ for some $h\ge3$ and $\beta\in(0,1)$. Then the first part of Theorem~\ref{teo:smooth covering} should hold with an $(n-1)$-dimensional manifold of class $C^{h+1,\alpha}$ for all $\alpha<\beta$.
    Similarly, Theorem~\ref{teo:generic R4} is expected to hold assuming $0<f\in C^\infty$.
\end{oss}

\subsection{Ideas of the proof}
The proof of Theorems~\ref{teo:generic R4} and \ref{teo:smooth covering} are based on a refinement of the results and the techniques used in \cite{FRS}. In particular, once Theorem~\ref{teo:smooth covering} is proven, to show Theorem~\ref{teo:generic R4} it is only necessary to control the size of the projection onto the time axis of $\Sigma\setminus\Sigma^\infty$.
Since for solutions in $\R^{4+1}$ this set is $2$-dimensional, it is sufficient to show a sharp ``quadratic cleaning'' of the free boundary, namely that at all singular points $(x_0,t_0)$
\begin{equation}\label{eq:intro quadratic cleaning}
    \partial\{u(\cdot,t_0+Cr^2)>0\}\cap B_r(x_0) =\emptyset\quad\forall r>0,
\end{equation}
and combine this with a covering argument.

We now illustrate the strategy adopted to prove Theorem~\ref{teo:smooth covering}. Let $(x_0,t_0)$ be a singular point, then
\begin{equation}\label{eq:first taylor}
    u(x_0+x,t_0+t) = p_{2,x_0,t_0} + o(|x|^2+|t|),
\end{equation}
where $p_{2,x_0,t_0}$ is a quadratic polynomial of the form $\frac12 x^t\cdot Ax$, with $A\ge0$ and $\tr A=1$.
In particular, the singular set can be stratified as $\Sigma=\cup_{k=0}^{n-1}\Sigma_k$, where
\[
    \Sigma_k\coloneqq\{\dim(\{p_{2,x_0,t_0}=0\})=k\}.
\]
One of the consequences of~\eqref{eq:first taylor} is that each stratum $\Sigma_k$ can be covered by a $k$-dimensional manifold of class $C^1$ in space and $C^{1/2}$ in time. Hence the main challenge to show Theorem~\ref{teo:smooth covering} is to prove that the stratum $\Sigma_{n-1}$ can be covered by a smooth (i.e. of class $C^\infty$ in space and time, with respect to the euclidean structure) $(n-1)$-dimensional manifold.
In order to increase the regularity of the covering manifolds, one needs to improve the expansion~\eqref{eq:first taylor} to higher order, thus Theorem~\ref{teo:smooth covering} is based on a pointwise smooth expansion at all points in the maximal stratum $\Sigma_{n-1}$.
Note that a solution of the Stefan problem is not smooth (as the optimal regularity for a solution of \eqref{eq:SP} is $C^{1,1}$ in space and $C^1$ in time), so one needs to introduce suitable functions (namely two-sided polynomial Ansätze, see Definition~\ref{defn:two sided ansatz}) to approximate the solution at all orders.

\begin{teo}\label{teo:smooth expansion}
For all $\rho,M,c>0$, $\alpha\in(0,1)$ and $k\ge3$ there exist $\bar r,\beta>0$ such that for all $u$ solving the Stefan problem \eqref{eq:SP} in $B_1\times[-1,1]$ and in the class $\cS(M,c,\rho)$ (see Definition~\ref{defn:non-degeneracy class}) and for all $(x_0,t_0)\in\Sigma_{n-1}(u)\cap B_{1-\rho}\times[-1+\rho^2,1]$ there is a two-sided polynomial Ansatz $\mathscr P_k$ (see Definition~\ref{defn:two sided ansatz}) such that
\[
    \|u(x_0+\cdot,t_0+\cdot) - \mathscr P_k\|_{L^2(B_r\times[-r^2,-r^{2+\beta}])} \le r^{k+\alpha}\quad \forall r<\bar r.
\]
\end{teo}
One of the many contributions of~\cite{FRS} consists in noting that to prove Theorem~\ref{teo:smooth expansion} it is actually sufficient to show a $C^{3+\beta}$ expansion, namely to prove that for all $(x_0,t_0)\in\Sigma_{n-1}$ there is a cubic two-sided polynomial ansatz $\mathscr P_3$ such that
\begin{equation}\label{eq:C3beta expansion introduction}
    u(x_0+x,t_0+t) = \mathscr P_3(x,t) + O((|x|+\sqrt{|t|})^{3+\beta}).
\end{equation}
As explained in \cite{FRS}, the main idea is that since $\partial_tu>0$ \eqref{eq:C3beta expansion introduction} implies that $u$ behaves like two regular solutions in the domain $\Omega^\beta\coloneqq\{|x|^{2+\beta}<-t\}$, whose scaling is subcritical with respect to the parabolic scaling.
The authors were able to show in \cite{FRS} a $C^\infty$ expansion at regular points solid enough to be applied in this context.

\medskip\noindent As pointed out in~\cite{FRS}, the proof of~\eqref{eq:C3beta expansion introduction} presents two main difficulties. The first one is showing the weaker expansion
\begin{equation}\label{eq:cubic expansion}
    u(x_0+x,t_0+t) = p_2(x) + O((|x|+\sqrt{|t|})^3),
\end{equation}
at all points in the maximal stratum.
This issue can be solved by proving the parabolic version of \cite{franceschinifrequencygap}.
This consists in establishing a frequency gap for the parabolic thin obstacle problem (see~\cite{DanielliPTOP} and references therein), excluding the existence of nontrivial (parabolically) $\lambda$-homogeneous solutions when $\lambda\in(2,3)$.

The most delicate part is improving the cubic expansion~\eqref{eq:cubic expansion} to the enhanced $C^{3+\beta}$ expansion \eqref{eq:C3beta expansion introduction}. This was done in \cite{FRS} ``at most points'', i.e. up to an $(n-2)$-dimensional set. Their approach, though, is not based on solid monotonicity formulas.
The main contribution of this paper is to simplify this step by fully exploiting the monotonicity of a cubic Weiss-type energy, showing an epiperimetric inequality at small enough scales at all points satisfying the cubic expansion \eqref{eq:cubic expansion}.\\
An epiperimetric inequality has been introduced for the first time in the context of minimal surfaces in~\cite{Reifenberg} and used in the context of a free boundary problem in~\cite{Weiss99}.
Since then, it has been extensively used in free boundary problems to show uniqueness and rate of convergence to blow-ups of solutions (see for instance \cite{ColomboSpolaorVelichkovObstacle,EngelsteinSpolaorVelichkov,CarducciEpiperimetric} and references therein for elliptic problems, and \cite{Shi20,ColomboSpolaorVelichkovParabolic} and references therein for parabolic problems).
We note in particular that a similar problem has been studied in \cite{CarducciEpiperimetric,SaYu}, namely an expansion of the type \eqref{eq:C3beta expansion introduction} at odd points in the Signorini problem.
However, to show the epiperimetric inequality (see Proposition~\ref{prop:epiperimetric stefan}) it seems best for our parabolic setting the approach used by~\cite{Shi20}.

Finally, we note that the techniques used here do not apply to lower strata of the singular set, since an expansion of type \eqref{eq:cubic expansion} is \textit{false} at all points of the lower strata (see \cite{FRS}). The study of those points becomes relevant for instance when trying to understand the size of the projection onto the time axis of the singular set of solutions in $5+1$ dimensions.

\subsection{Structure of the paper} In Section~\ref{sec:frequency gap} we show a frequency gap for the parabolic thin obstacle problem. In Section~\ref{sec:C3beta expansion} we prove a $C^{3+\beta}$ expansion at singular points in the maximal stratum, and in Section~\ref{sec:main results} we use this result to show Theorem~\ref{teo:smooth expansion}. In Section~\ref{sec:generic} we prove Theorem~\ref{teo:generic R4}.

\subsection{Acknowledgements} The author is grateful to Alessio Figalli for pointing out the usefulness of \cite[Lemma A.3]{FerryLemma}, as well as for his encouragement and supervision during this project.

\section{Notation}\label{sec:notation}
Given a point $(x_0,t_0)\in\R^{n+1}$ and $r>0$ we define the parabolic cylinder
\[
    C_r(x_0,t_0) = B_r(x_0)\times(t_0-r^2,t_0].
\]
We will write $C_r$ when $(x_0,t_0)=(0,0)$.
We will denote the heat operator by
\[
    \Heat = \Delta-\partial_t
\]
and the generator of parabolic dilations by
\[
    Z = x\cdot\nabla +2t\partial_t.
\]
Also, for $\lambda\in\R$, we will denote the Ornstein-Uhlenbeck operator by
\[
    \cL = \Delta-\frac x2\cdot\nabla,\quad \cL_\lambda= \cL + \frac\lambda2.
\]
We define the ``reverse heat kernel'' as
\begin{equation}\label{eq:gaussian kernels}
    G_n(x,t) = \frac1{(4\pi|t|)^{n/2}}e^{\frac{|x|^2}{4t}}\text{ for }t<0,\quad G_n=G_n(\cdot,-1)
\end{equation}
and the gaussian measure on $\R^n$ as
\[
    d\gamma_n = G_n(\cdot,-1)dx = (4\pi)^{-n/2}e^{-\frac{|x|^2}{4}} dx.
\]
We will denote $L^2(\gamma_n)=L^2(\R^n,\gamma_n)$ and
\[
    H^1(\gamma_n)=\{f\in H^1_\loc(\R^n,\cL^n)\,:\,f,\nabla f\in L^2(\gamma_n)\}.
\]
Finally, through the paper we will denote by $\zeta$ a smooth spatial cutoff function satisfying
\begin{equation}\label{eq:spatial cutoff}
    \zeta\in C^\infty_c(B_{1/2}),\quad0\le\zeta\le1,\quad \zeta\equiv 1\text{ in }B_{1/4}.
\end{equation}

\section{Frequency gap in the parabolic Signorini}\label{sec:frequency gap}

    In this section we show a frequency gap for the parabolic Signorini problem. As a consequence, the only possible frequency at a singular point in the maximal stratum of a solution of the Stefan problem is 3.
    A Lipschitz function $q\colon\R^n\times(-\infty,0)$ solves the parabolic Signorini problem in $\R^n$ with zero obstacle provided
    \begin{equation}\label{eq:ptop}\begin{cases}
    q\mathbf Hq=0    &\text{in }\R^n\times(-\infty,0];\\
    \mathbf Hq=0            &\text{in }\R^n\times(-\infty,0]\setminus\{x_n=0,q=0\};\\
    q\ge0\text{ and }\partial_n q\le0   &\text{in }\{x_n=0\}\times(-\infty,0],
    \end{cases}\end{equation}
    where $\Heat=\Delta-\partial_t$ denotes the heat operator. 
    We say that $q$ is (parabolically) $\lambda$-homogeneous for $\lambda\in\R$ if
    \[
        q(r x,r^2 t)=r^\lambda q(x,t)\quad\text{for all }x\in\R^n, t<0.
    \]
    The main result of this section is the following. It is the parabolic equivalent of~\cite[Theorem 1]{franceschinifrequencygap}.
    \begin{teo}[Frequency gap]\label{teo:frequency gap}
        Let $q$ be a $\lambda$-homogeneous solution of~\eqref{eq:ptop} for some $\lambda\in(2,3)$ satisfying
        $\int_{\{t=-1\}}(q^2+|\nabla q|^2)d\gamma_n<+\infty$.
        Then $q\equiv0$.
    \end{teo}
    A function $q$ is a (parabolically) $\lambda$-homogeneous solution of~\eqref{eq:ptop} if and only if
    \[\begin{split}
        \cL_\lambda q(\cdot,-1)=0\quad\text{in }\R^n\setminus\{q(\cdot,-1)=0,x_n=0\},\\ q(\cdot,-1)\ge0\quad \text{and}\quad\partial_nq(\cdot,-1)\le0\quad\text{in }\{x_n=0\},
    \end{split}\]
    where $\cL$ is defined in Section~\ref{sec:notation}.
    We will need the following Lemma, whose proof is postponed at the end of this Section.
    \begin{lem}\label{lem:test function}
        For all $2\le\lambda\le3$ there is a unique (up to a scalar multiple) non-trivial solution $p_\lambda(t)$ of
        \begin{equation}\label{eq:eigenfunctions}
            p_\lambda'' -\frac t2 p_\lambda' + \frac\lambda2 p_\lambda =0 \quad\text{in }(0,+\infty)
        \end{equation}
        such that $\int_0^{+\infty} (p_\lambda'^2+p_\lambda^2)\,d\gamma_1<+\infty$.
        Moreover, for all $2<\lambda<3$ we have $p_\lambda'(0) \cdot p_\lambda(0)>0$.
    \end{lem}
    Using the lemma above, we can prove Theorem~\ref{teo:frequency gap} following the argument in \cite{franceschinifrequencygap}.
    \begin{proof}[Proof of Theorem~\ref{teo:frequency gap}]
    In $\R^n_+=\{x_n>0\}$ we consider $p_\lambda(x)=p_\lambda(x_n)$ a non-trivial solution of $\cL
    _\lambda p_\lambda=0$ in $\{x_n>0\}$ given by Lemma~\ref{lem:test function}.
    Writing $G_n=G_n(\cdot,-1)$, as $\nabla G_n = -\tfrac x2G_n$ for any functions $f,g$ we have $G_nf\cL g=f\divop(G_n\nabla g)$.
    Thus an integration by parts gives
    \begin{equation}\label{eq:frequency gap integration by parts}\begin{split}
        0&=\int_{\{x_n>0\}} (p_\lambda \cL_\lambda q-q\cL_\lambda p_\lambda)\,d\gamma_n= \int_{\{x_n>0\}} (p_\lambda \cL q-q\cL p_\lambda)\,d\gamma_n \\
        &=\int_{\{x_n>0\}}p_\lambda\divop(G_n\nabla q)-q\divop(G_n\nabla p_\lambda)\,dx\\
        &=\int_{\{x_n>0\}} \divop(p_\lambda G_n\nabla q-qG_n\nabla p_\lambda) \,dx\\
        &=\int_{\{x_n=0\}} (q\partial_np_\lambda - p_\lambda\partial_n q)G_n\,dx
    \end{split}\end{equation}
    which implies
    \[
        p_\lambda'(0)\int_{\{x_n=0\}}q\,G_ndx=p_\lambda(0)\int_{\{x_n=0\}}\partial_nq\,G_ndx.
    \]
    Since $q\ge0$ and $\partial_nq\le0$ in $\{x_n=0\}$, Lemma~\ref{lem:test function} implies that for $2<\lambda<3$ the two terms have opposite sign, unless $\partial_nq\equiv q\equiv0$ in $\{x_n=0\}$. By unique continuation this implies $q\equiv0$, as we wanted.
\end{proof}

Before proving Lemma~\ref{lem:test function}, we recall spectral properties of the Ornstein-Uhlenbeck operator with Dirichlet boundary conditions in unbounded intervals in $\R$.
Given $a\in\R$ we set $I_a\coloneqq (a,+\infty)$ and we define
\[\begin{split}
    L^2(I_a,\gamma_1)&\coloneqq\left\{f\colon I_a\to\R\colon \int_a^{+\infty}f^2\,d\gamma_1<+\infty\right\},\\ H^1_0(I_a,\gamma_1)&\coloneqq\left\{f\colon I_a\to\R\colon \int_a^{+\infty}(f^2+f'^2)\,d\gamma_1<+\infty,f(a)=0\right\}.
\end{split}\]
We recall the Ornstein-Uhlenbeck operator
\[
    -\cL u= -G_1^{-1} (G_1 u')'=-u''+\frac t2 u'.
\]
For $a\in\R$ and $f\in L^2(I_a,G_1)$ we denote by $u=(-\cL)_a^{-1}f$ the solution (provided it exists and is unique) $u\in H^1_0(I_a,G_1)$ of
\[\begin{cases}
    -\cL u=f &\text{in }I_a,\\
    u(x)=0  &\text{for }x=a.
\end{cases}\]

We will need the following properties of the operators $\cL_a^{-1}$ for $a\in[\sqrt 2,\sqrt 6]$, which are a consequence of the Gaussian log-Sobolev and Poincar\'e inequalities.
\begin{lem}\label{lem:compact operator}
    The operator $(-\cL)_a^{-1}\colon L^2(I_a,\gamma_1)\to L^2(I_a,\gamma_1)$ is well defined, compact, and self-adjoint for all $a\ge0$.
    Thus, its spectrum is discrete.
\end{lem}
\begin{proof}
By the Gaussian Poincar\'e inequality (see \cite[Theorem 4.6.3]{BakryMarkov}) there is $C>0$ such that
\[
\int_\R f^2d\gamma_1-\left(\int_\R fd\gamma_1\right)^2\le C\int_\R f'^2d\gamma_1\quad \text{for all }f\in H^1(\R,\gamma_1).
\]
Given $a\ge0$ and $f\in H^1_0(I_a,\gamma_1)$ we can apply this inequality to the extension $\tilde f\in H^1(\R,\gamma_1)$ given by
\begin{equation}\label{eq:extension proof spectrum}
    \tilde f (x) = \begin{cases}
        0   &\text{for }|x|<a,\\
        f(x)    &\text{for }x>a,\\
        -f(-x)  &x<-a.
    \end{cases}
\end{equation}
Since $\int_\R \tilde fd\gamma_1=0$, we find
\[
\int_a^{+\infty}f^2d\gamma_1\le C\int_a^{+\infty}f'^2d\gamma_1.
\]
It follows by the Lax-Milgram Theorem (see for instance~\cite[Corollary 5.8]{Brezis}) that for all $f\in L^2(I_a,\gamma_1)$ there is a unique solution $u\in H^1_0(I_a,\gamma_1)$ of
\[
\begin{cases}
    -\cL u=f&\text{in }I_a,\\
    u(a)=0
\end{cases}
\]
and satisfying
\[
    \int_a^{+\infty} u'^2d\gamma_1\le C\int_a^{+\infty}f^2d\gamma_1
\]
for some $C>0$ (independent from $f$).
Thus the operator $(-\cL)_a^{-1}\colon L^2(I_a,\gamma_1)\to H_0^1(I_a,\gamma_1)$ is well defined, bounded and symmetric.
If we show that the embedding $H^1_0(I_a,\gamma_1)\hookrightarrow L^2(I_a,\gamma_1)$ is compact then $(-\cL)_a^{-1}\colon L^2(I_a,\gamma_1)\to L^2(I_a,\gamma_1)$ will be compact.
Standard arguments for compact symmetric operators will then allow us to conclude (see e.g.~\cite[Theorem 6.11]{Brezis}).

Since $\gamma_1$ is locally equivalent to the standard Lebesgue measure,
the embeddings $H^1_0(I_a\cap B_R,\gamma_1)\hookrightarrow L^2(I_a\cap B_R,\gamma_1)$ are compact for all $R>0$.
Thus, it is sufficient to show that for all $M,\eps>0$ there is $R>0$ so that
\[
    \int_{\{|x|>R\}} f^2d\gamma_1\le \eps
\]
for all $f\in H^1(\R,\gamma_1)$ satisfying
\[
    \int_\R (f^2+f'^2)d\gamma_1\le M,
\]
as given $f\in H^1_0(I_a,\gamma_1)$ we can apply this to the extension $\tilde f\in H_0^1(\R,\gamma_1)$ defined in~\eqref{eq:extension proof spectrum} together with a standard diagonal argument.
Recall the Gaussian log-Sobolev inequality (see \cite[Proposition 5.5.1]{BakryMarkov})
\[
\int_\R F^2\log F^2d\gamma_1\le C\int_\R F'^2d\gamma_1+\left(\int_\R F^2d\gamma_1\right)\log\left(\int_\R F^2d\gamma_1\right).
\]
Fix $\lambda,R>0$ large to be set later and consider $F=\max\{|f|,1\}$. Since $F=|f|$ on $\{f^2>\lambda\}$, $1\le F^2\le 1+f^2$ and $|F'|\le|f'|$ we find
\[\begin{split}
    \log(\lambda^2)\int_{\{|x|>R\}} f^2\chi_{\{f^2>\lambda\}}d\gamma_1&\le\int_\R F^2\log F^2\chi_{\{f^2>\lambda\}}d\gamma_1\le \int_\R F^2\log F^2d\gamma_1\\
    &\le C\int_\R F'^2d\gamma_1 + \left(\int_\R F^2d\gamma_1\right) \log\left(\int_\R F^2d\gamma_1\right)\\
    &\le C\int_\R f'^2d\gamma_1 + \left(1+\int_\R f^2d\gamma_1\right)\log\left(1+\int_\R f^2 d\gamma_1\right)\\
    &\le CM + (1+M)\log(1+M).
\end{split}\]
Thus,
\[\begin{split}
    \int_{\{|x|>R\}} f^2d\gamma_1 &\le \lambda\int_{\{|x|>R\}}d\gamma_1 + \int_{\{|x|>R\}} f^2\chi_{\{f^2>\lambda\}}d\gamma_1\\
    &\le \lambda \int_{\{|x|>R\}}d\gamma_1 +\frac1{\log \lambda} C(M).
\end{split}\]
Choosing first $\lambda$ so that $(\log\lambda)^{-1}C(M)<\eps/2$ and then $R$ so that $\lambda\int_{\{|x|>R\}}d\gamma_1<\eps/2$ we find the claim.
\end{proof}

\begin{proof}[Proof of Lemma~\ref{lem:test function}]
\textit{Existence.} By Lemma~\ref{lem:compact operator} the operator $(-\cL)_a^{-1}$ is compact and symmetric for all $\sqrt 2\le a\le\sqrt6$. In particular there is a first eigenvalue $\lambda_1(a)>0$ and functions $p_a$ with finite energy that are positive on $I_a$ and solve
\[
    p_a''-\frac t2p_a' + \frac{\lambda_1(a)}{2}p_a=0\quad\text{in }I_a,\quad p_a(a)=0.
\]
Since the functions
\begin{equation}\label{eq:proof eigenfunctions polynomials}
    p_2(t)=c_2(t^2-2),\quad p_3(t)=c_3t(t^2-6),\quad c_2,c_3>0
\end{equation}
solve
\[
    -(G_1 u')'=\tfrac\lambda2 G_1u\quad\text{in }(a,+\infty),\quad u(a)=0
\]
with $a=\sqrt 2,\sqrt 6$ and $\lambda=2,3$ and are positive in $I_{\sqrt 2}$ and $I_{\sqrt6}$ respectively,
it follows that $\lambda_1(\sqrt 2)=2$ and $\lambda_1(\sqrt 6)=3$.
As the first eigenvalue depends continuously on $a$, for all $2\le \lambda\le3$ there is $a_\lambda\in[\sqrt 2,\sqrt 6]$ such that $\lambda_1(a_\lambda)=\lambda$.
Existence of solutions of~\eqref{eq:eigenfunctions} with finite energy follows by extending $p_{a_\lambda}$ to $\R$ by solving a Cauchy problem for the ODE.

\smallskip\noindent\textit{Uniqueness.} Assume that $p$ is a finite energy solution of~\eqref{eq:eigenfunctions} for some $2\le\lambda\le3$, let $a$ be such that $\lambda_1(I_a)=\lambda$ and denote by $p_\lambda$ the first eigenfunction on $I_a$.
Then, we compute as in~\eqref{eq:frequency gap integration by parts}
\[
    0=\int_a^{+\infty}(p\cL_\lambda p_\lambda-p_\lambda\cL_\lambda p)\,d\gamma_1= G_1(a)(p_\lambda(a)p'(a)-p_\lambda'(a)p(a)).
\]
Since $p_\lambda(a)=0$ and $p_\lambda'(a)\neq0$ (otherwise $p_\lambda\equiv0$) this implies that also $p(a)=0$. It follows that there is $c\in\R$ such that both $cp$ and $p_\lambda$ solve~\eqref{eq:eigenfunctions} with $cp(a)=p_\lambda(a)=0$ and $cp'(a)=p_\lambda'(a)$, which implies $p\equiv cp_\lambda$ as we wanted.

\smallskip\noindent\textit{Sign condition.} The family of solutions $p_\lambda$, $2\le\lambda\le3$, chosen so that $\int_0^{+\infty}p_\lambda^2\,d\gamma_1=1$ and $p_\lambda>0$ for $t>\sqrt6$ depends smoothly on $\lambda$. Moreover, for $\lambda=2,3$ they coincide with the functions in~\eqref{eq:proof eigenfunctions polynomials} for some $c_2,c_3>0$ chosen so that $\int_0^{+\infty} p_2^2d\gamma_1=\int_0^{+\infty} p_3^2d\gamma_1=1$.
We note in addition that if $p_\lambda(0)=0$ then an odd reflection of $p_\lambda$ will be an odd eigenfunction of $\cL$ in the whole $\R$, which implies that $\lambda$ is an odd integer. Similarly, $p_\lambda'(0)=0$ implies $\lambda$ is an even integer.
Considering now the function $\lambda\mapsto p_\lambda(0)$, since it is continuous in $\lambda$, vanishes exactly when $\lambda=3$ and $p_2(0)<0$ it follows that $p_\lambda(0)<0$ for all $2\le\lambda<3$.
Similarly, since $p_\lambda'(0)$ depends continuously on $\lambda$, vanishes exactly for $\lambda=2$ and $p_3'(0)<0$, we also find $p_\lambda'(0)<0$ for all $2<\lambda\le3$, as we wanted.
\end{proof}

\section{Pointwise $C^{3+\beta}$ expansion in the maximal stratum}\label{sec:C3beta expansion}

We consider solutions $u$ of \eqref{eq:SP} in $B_1\times[-1,1]$ with uniform nondegeneracy of the time derivative at singular points of the maximal stratum.
\begin{defn}\label{defn:non-degeneracy class}
    Given $M,c>0$, $\rho\in(0,1)$ and $u\colon B_1\times[-1,1]\to[0,+\infty)$, we say that $u\in \cS(M,c,\rho)$ if it solves \eqref{eq:SP} and it satisfies:
    \begin{enumerate}[label=\roman*)]
        \item $|u| \le M$ in $B_{1-\rho/2}\times[-1+\rho^2/4,1]$;
        \item for all $(x_0,t_0)\in \Sigma_{n-1}(u)\cap B_{1-\rho}\times[-1+\rho^2,1]$ and all $r<\rho/100$ we have
        \begin{equation}\label{eq:condition nondegeneracy partial_t}
            \fint_{C_r(x_0,t_0)}\partial_t u \ge cr,
        \end{equation}
        where $C_r(x_0,t_0)=B_r(x_0)\times (t_0-r^2,t_0]$.
    \end{enumerate}
\end{defn}
We note the following consequence of \cite[Lemma 8.4]{FRS} together with local $L^\infty$ bounds for solutions of \eqref{eq:SP}.
\begin{lem}\label{lem:non-degeneracy time derivative single solution}
    Let $u\colon B_1\times[-1,1]\to[0,+\infty)$ solve \eqref{eq:SP}. Then, for all $\rho\in(0,1)$ there are $M,c>0$ so that $u\in\cS(M,c,\rho)$.
\end{lem}
The main result of this section is the following.
\begin{teo}\label{teo:C3beta expansion}
    Given $M,c,\rho>0$ there are $C_0>0$ and $\beta\in(0,1/2)$ such that the following holds:
    
    Let $u\in \cS(M,c,\rho)$ (see Definition~\ref{defn:non-degeneracy class}) and let
    $(x_0,t_0)\in \Sigma_{n-1}\cap B_{1-\rho}\times[-1+\rho^2,1-\rho^2]$. Then there is a parabolically 3-homogeneous function $p_3$ solving the parabolic thin obstacle problem \eqref{eq:ptop} such that, up to a rotation in space,
    \[
        |u(x_0+r\cdot,t_0+r^2\cdot)-\tfrac12r^2x_n^2-r^3p_3|\le C_0 r^{3+\beta}\quad\text{in }C_1\quad\forall r\in(0,\rho).
    \]
\end{teo}
We first collect some useful properties.
Given $f\colon\R^n\times(-1,0)\to\R$ we set
\begin{equation}\label{eq:weiss}
     W_3(f,r) = r^{-6}\int_{\{t=-r^2\}} (r^2|\nabla v|^2-\tfrac32v^2)G_n(x,t)dx,\quad r\in(0,1),
\end{equation}
where we recall the reversed heat kernel \eqref{eq:gaussian kernels}.
If $u\colon B_1\times(-1,1)\to[0,+\infty)$ solves \eqref{eq:SP} with
$(0,0)\in\Sigma$ and recalling the cutoff $\zeta$ defined in~\eqref{eq:spatial cutoff}, then \cite[Lemma 5.3]{FRS} yields
\begin{equation}\label{eq:almost mononicity weiss}
    \partial_r W_3(\zeta(u-p_2),r) \ge \frac1{r^7}\int_{\{t=-r^2\}} (Z(\zeta(u-p_2))-3\zeta(u-p_2))^2d\gamma_n - Ce^{-\frac1r}\quad\forall r\in(0,\tfrac12),
\end{equation}
for some $C>0$ depending only on $n, \zeta$ and $\|u(\cdot,0)\|_{L^\infty(B_1)}$, where $Z$ is defined in Section~\ref{sec:notation}.\\
Finally, since $\nabla G_n = -\frac y2G_n$, we have $f\divop(G_n\nabla g) = G_nf\cL g$ for all $f,g$ sufficiently regular. Thus, an integration by parts yields
\begin{equation}\label{eq:integration by parts}
        \int_{\R^n}\nabla f\cdot\nabla g\,d\gamma_n =
        - \int_{\R^n}f\cL g\,d\gamma_n.
\end{equation}

\subsection{An epiperimetric inequality} To show Theorem~\ref{teo:C3beta expansion}, we prove an epiperimetric inequality for \eqref{eq:weiss}. We will work in conformal coordinates $(y,s)$ given by
\begin{equation}\label{eq:conformal coordinates}
    (x,t) = (e^{-s/2}y,-e^{-s}).
\end{equation}
Given $f\colon \R^n\times(-1,0)$ we define $\tilde f\colon \R^n\times(0,+\infty)$ by
\begin{equation}\label{eq:conformal transformation}
    \tilde f(y,s) = e^{3s/2}f(x(y,s),t(s)).
\end{equation}
Notice that given $r>0$ and $f\colon\R^n\times(-\infty,0)$,
the parabolic 3-homogeneous rescaling
\[
    f_{3,r}(x,t)\coloneqq r^{-3}f(rx,r^2t)
\]
corresponds to a translation in time
\[
    \tilde f_{3,r}=\tilde f(\cdot,\cdot-2\log r).
\]
Moreover, if we set $v=u-p_2$ then
\begin{equation}\label{eq:conformal Stefan}
    (\cL_3-\partial_s)\tilde v = - e^{s/2}\chi_{\{\tilde u=0\}}\quad\text{in }B_{e^{s/2}}\times[0,+\infty),
\end{equation}
where $\cL_3$ is defined in Section~\ref{sec:notation}.\\
We define
\[
    \cP_3^+\coloneqq\{p \,:\, \cL_3p = 0 \text{ in }\R^n\setminus\{y_n=0\}\text{ and }\cL_3p\le0\text{ in }\R^n,
    p\equiv0\text{ in }\{y_n=0\}\}.
\]
We note that $p\in\cP_3^+$ if and only if $p=q(\cdot,-1)$ for some $q$ 3-homogeneous solution of \eqref{eq:ptop}.
It follows from \cite[Lemma 9.2]{FRS} that the 3-homogeneous extension in $\R^n\times(-\infty,0)$ of any $p\in \cP_3^+$ will be of the form
\begin{multline}\label{eq:cubic blow-ups}
    p(x',x_n,t) = (x_n)_+ (\tfrac {a^+}6x_n^2 + (a^+-b^+)t - \tfrac12B^+x'\cdot x')\\
    +(x_n)_- (\tfrac {a^-}6x_n^2 + (a^--b^-)t - \tfrac12B^-x'\cdot x'),
\end{multline}
where the parameters $a^\pm,b^\pm\in\R$, $B^\pm\in \mathrm{Sym}(n-1,\R)$ satisfy
\[
    B^++B^-\ge0,\quad b^\pm=\tr B^\pm\quad\text{and}\quad a^\pm\ge b^\pm
\]
and we write $x=(x',x_n)\in\R^{n-1}\times\R$ for all $x\in\R^n$.
We also note that the vector space generated by $\cP_3^+$ is
\[
    \cP_3 \coloneqq \{p\,:\, \cL_3 p = 0 \text{ in }\R^n\setminus\{y_n=0\}\text{ and }p\equiv0\text{ in }\{y_n=0\}\}.
\]
Setting $w= \zeta(u-p_2)$, in conformal coordinates the Weiss energy~\eqref{eq:weiss} corresponds to
\begin{equation}\label{eq:conformal weiss}
    W_3(\tilde{w},s) \coloneqq \int_{\R^n}(|\nabla \tilde{w}|^2-\tfrac32 \tilde{w}^2)d\gamma_n
\end{equation}
and the almost monotonicity formula \eqref{eq:almost mononicity weiss} reads
\begin{equation}\label{eq:conformal monotonicity weiss}
    \|\partial_s \tilde w(\cdot,s)\|_{L^2(\gamma_n)}^2 \le -\tfrac12\partial_s W_3(\tilde w,s) + C \exp(-e^{s/2})\quad\forall s>-2\log 2.
\end{equation}
Since we will need to absorb some higher order terms in $s$,
we prove an epiperimetric inequality for the following modified Weiss energy:
\begin{equation}\label{eq:modified weiss}
    \widetilde W_3(\tilde w,s) \coloneqq W_3(\tilde w,s) + e^{-s/2}.
\end{equation}

\begin{prop}[Epiperimetric inequality]\label{prop:epiperimetric stefan}
    For all $c', M>0$ there are constants $\eps_0,\delta_0,s_0>0$ depending only on $c',M,n$ such that the following holds:

    Let $u$ solve \eqref{eq:SP} in $B_1\times[-1,0]$ with $|u|\le M$ and $(0,0)\in\Sigma_{n-1}$ with blow-up $p_2$, and set $v=u-p_2$ and $w=\zeta v$ where $\zeta$ is as in \eqref{eq:spatial cutoff}.
    Assume in addition that there are $\sigma>s_0$ and $p_3\in\cP_3^+$ with $\partial_tp_3^\even\ge c'|x_n|$ such that
    \[
        \int_\sigma^{\sigma+1}\|\tilde w-p_3\|^2_{L^2(\gamma_n)}ds\le \delta_0^2.
    \]
    Then
    \[
        \widetilde W(\tilde w,\sigma+1) \le (1-\eps_0)\widetilde W(\tilde w,\sigma).
    \]
\end{prop}

We will make use of the following results.
\begin{lem}[{\cite[Lemma 10.2]{FRS}}]\label{lem:barrier}
    Let $u$ be a bounded solution of \eqref{eq:SP} in $B_1\times[-1,1]$. Given positive constants $c_0, C_0,r_0,R_0$, there is $\delta>0$ such that the following holds: Let $Q_2$ be any 2-homogeneous caloric polynomial satisfying $\Heat (x_nQ_2)=0$ and $\|Q_2\|_{L^2(C_1)}\le C_0$. Assume that
    \[
        \frac{(u-p_2)_r}{r^3}\le c_0 |x_n|t + C_0|x_n|^3 + Q_2(x')x_n+\delta\quad\text{in }\{-4\le t\le-1,|x|\le 2R_0\}
    \]
    for all $r\in(0,r_0)$.
    Then
    \[
        u(r\cdot,r^2\cdot)\le Cr^4\quad\text{in }\{x_n=0,|x|\le R_0,t=-1\}
    \]
    for all $r\in(0,r_0)$.
\end{lem}

\begin{lem}[{\cite[Proposition 6.6]{FRS}}]\label{lem:closeness contactSet}
    Let $u \colon B_1\times(-1,1)\to[0,\infty)$ be a bounded solution of \eqref{eq:SP}, $(0,0)\in\Sigma_{n-1}$, and set $w\coloneqq u-p_2$.
    Then
    \[
        \{u(\cdot,t) = 0\}\cap B_r \subset \{x \mid \dist(x,\{p_2 = 0\}) \le Cr^2\}
    \]
    for all $r\in(0,1/2)$ and $t\ge-r^2$. In addition, the constant $C$ depends only on $n$ and $\|u(\cdot,0)\|_{L^\infty(B_1)}$.
\end{lem}
\begin{lem}[{\cite[Theorem 4.16]{Wang}}]\label{lem:parabolic ABP}
    Let $w$ be such that either $(\cL_3-\partial_s) w\ge0$ or $\Heat w\ge0$ in $C_1$. Then
    \[
        \sup_{C_{1/2}} w \le C \left(\int_{C_1}w_+^2\right)^{1/2}
    \]
    for some dimensional $C>0$.
\end{lem}

\begin{oss}\label{oss:sign wDeltaw}
    If $\tilde v$ solves \eqref{eq:conformal Stefan}, then
    \[
        \tilde v(\cL_3-\partial_s) \tilde v = \tilde p\chi_{\{\tilde u=0\}}\ge 0.
    \]
\end{oss}
\begin{lem}[\cite{Caffarelli77}]
    Let $u\colon B_1\times[-1,1]\to[0,+\infty)$ be a bounded solution of \eqref{eq:SP} with $u(0,0)=0$. Then
\begin{equation}\label{eq:local C11 estimates}
    \sup_{B_{1/2}\times[-1/2,0]} |D^2u|+|\partial_tu|\le C\|u(\cdot,0)\|_{L^\infty(B_1)}.
\end{equation}
\end{lem}

We will also need to control the errors introduced in the equations by multiplying with the cutoff $\zeta$ defined in \eqref{eq:spatial cutoff}, which corresponds to \cite[Lemma 5.2]{FRS} in conformal coordinates.
Setting $\zeta_\sigma(y) = \zeta(e^{-\sigma/2}y)$, given $v\colon B_{e^{\sigma/2}}\times(0,+\infty)$ sufficiently regular, we note that
\[
    \widetilde{\zeta v}(\cdot,\cdot+\sigma) = \zeta_\sigma \tilde v(\cdot,\cdot+\sigma).
\]
\begin{lem}\label{lem:exponential errors}
If $v=u-p_2$ where $u$ solves~\eqref{eq:SP} in $B_1\times[-1,1]$ and $\tilde w=\widetilde{\zeta v}(\cdot,\cdot+\sigma)$, then
\begin{multline}\label{eq:exponential errors}
    \|\cL_3\tilde w(\cdot,0)-\zeta_\sigma \cL_3\tilde v(\cdot,\sigma)\|_{L^2(\gamma_n)}+\|\partial_s\tilde w(\cdot,0)-\zeta_\sigma\partial_s\tilde v(\cdot,\sigma)\|_{L^2(\gamma_n)}\\
    \le C \exp(-e^{\sigma/2})\quad\forall \sigma>-2\log 2.
\end{multline}
for some constant $C>0$ depending only on $n, \|u(\cdot,0)\|_{L^\infty(B_1)}$ and $\zeta$.
\end{lem}
\begin{proof}
    We compute
    \[
        |\partial_s (\zeta_\sigma\tilde v(\cdot,\sigma)) - \zeta_\sigma\partial_s\tilde v(\cdot,\sigma)|\le |\tilde v(\cdot,\sigma) \partial_s\zeta_\sigma|
    \]
    and
    \[
        |\cL_3(\zeta_\sigma\tilde v)-\zeta_\sigma\cL_3\tilde v| \le |\tilde v\Delta \zeta_\sigma | + |\tilde vy\cdot\nabla \zeta_\sigma|+ |\nabla\tilde v(\cdot,\sigma)\cdot\nabla\zeta_\sigma|.
    \]
    Thus
    \[
        \int_{\R^n} \left|\partial_s (\zeta_\sigma\tilde v(\cdot,\sigma)) - \zeta_\sigma\partial_s\tilde v(\cdot,\sigma)\right|^2d\gamma_n\le \int_{\R^n}\tilde v(\cdot,\sigma)^2|\partial_s\zeta_\sigma|^2G_n(\cdot,-1)dx
    \]
    and
    \[
        \int_{\R^n} \left|\cL_3(\zeta_\sigma\tilde v)-\zeta_\sigma\cL_3\tilde v\right|^2d\gamma_n\le C\int_{\R^n}\tilde v(\cdot,\sigma)^2(|\Delta \zeta_\sigma |^2+|y\nabla \zeta_\sigma|^2)+|\nabla\tilde v(\cdot,\sigma)\cdot\nabla\zeta_\sigma|^2)G_n(\cdot,-1)dx.
    \]
    Note that \eqref{eq:spatial cutoff} yields $\zeta_\sigma\equiv1$ in $B_{e^{\sigma/2}/4}$ and $\zeta_\sigma\equiv0$ on $B_{e^\sigma/2}$. Moreover, $e^{(3+n)\sigma/2}G_n\le C\exp(-e^{\sigma}/100)$ in $\R^n\setminus B_{e^{\sigma/2}/4}$.
    Finally, recalling the relation between $\tilde v$ and $u$, \eqref{eq:local C11 estimates} yields
    \[
        |\nabla \tilde v(\cdot,\sigma)|+|\tilde v(\cdot,\sigma)|\le C e^{3\sigma/2}\quad \text{in }B_{\frac12e^{\sigma/2}}
    \]
    for some $C>0$ depending only on $n,\|u(\cdot,0)\|_{L^\infty(B_1)}$.
    Thus
    \begin{multline*}
        \|\partial_s (\zeta_\sigma\tilde v(\cdot,\sigma)) - \zeta_\sigma\partial_s\tilde v(\cdot,\sigma)\|^2_{L^2(\gamma_n)}
        \le C\|\tilde v(\cdot,\sigma)\|^2_{L^\infty(\frac12B_{e^{\sigma/2}})} \|\partial_s\zeta_\sigma\|^2_{\infty}e^{n\sigma/2}\exp(-e^{\sigma}/4)\\
        \le C\exp(-e^{\sigma/2}),
    \end{multline*}
    and similarly for the other term, as we wanted.
\end{proof}

We will also need the following compactness result.
\begin{lem}\label{lem:cpt}
    Let $f_k$ be such that
    \[
        \int_0^1(\|f_k(s)\|^2_{H^1(\gamma_n)}+\|\partial_s f_k\|_{L^2(\gamma_n)}^2)ds\le C.
    \]
    Then up to a subsequence
    \[
        f_k\to f_\infty\quad\text{strongly in }L^2((0,1);L^2(\gamma_n)).
    \]
\end{lem}
\begin{proof}
    As a consequence of the gaussian log-Sobolev inequality, $H^1(\gamma_n)$ is compactly embedded in $L^2(\gamma_n)$ (the proof is identical to the one given in Lemma~\ref{lem:compact operator}). Then the Aubin-Lions compactness Theorem \cite{AubinLions} yields the desired result.
\end{proof}

\begin{proof}[Proof of Proposition~\ref{prop:epiperimetric stefan}]
    We argue by contradiction. Assume there are $u_k$ solving \eqref{eq:SP} in $B_1\times[-1,0]$ with $|u_k|\le M$, $\sigma_k\to+\infty$ and $p_k\in\cP_3^+$ satisfying $\partial_tp_k\ge c'|x_n|$ such that, setting
    \[
        v_k = u-p_2,\quad \tilde w_k =\widetilde{\zeta v_k}(\cdot,\cdot+\sigma_k),
    \]
    they satisfy
    \[
        \int_0^1\|\tilde w_k-p_k\|^2_{L^2(\gamma_n)}ds\le \frac1k\quad\text{and}\quad \sigma_k\ge k
    \]
    but
    \[
        \widetilde W(\tilde w_k(\cdot,\cdot-\sigma_k),\sigma_k+1) \ge (1-\tfrac1k)\widetilde W(\tilde w_k(\cdot,\cdot-\sigma_k),\sigma_k).
    \]
    In particular, this together with~\eqref{eq:conformal monotonicity weiss} implies
    \begin{equation}\label{eq:constant weiss}
        W(\tilde w_k,0)-W(\tilde w_k,1) + \frac14e^{-\sigma_k/2}\le \frac2k \int_0^1W(\tilde w_k,s)ds\quad\text{for }k\gg1.
    \end{equation}
    Indeed, by definition of $\widetilde W$ we find
    \[
        W_3(\tilde w_k,0) - W_3(\tilde w_k,1) + e^{-\sigma_k/2} (1-\tfrac1k-e^{-1/2})\le \tfrac 1kW_3(\tilde w_k,0)
    \]
    On the other hand,
    \[
        W_3(\tilde w_k,0)\le \frac{k}{k-1} W_3(\tilde w_k,1) +e^{-(\sigma_k+1)/2}(\tfrac{k}{k-1}-e^{1/2})\le \frac{k}{k-1}W_3(\tilde w_k,1)\quad\text{for }k\gg1
    \]
    Finally, \eqref{eq:conformal monotonicity weiss} implies
    \[
        W_3(\tilde w_k,1)\le W_3(\tilde w_k,s) + \tfrac 1{10}e^{-\sigma_k/2}\quad\forall s\in (0,1)\quad\text{for }k\gg1,
    \]
    hence \eqref{eq:constant weiss} follows.\\
    For each $k$ take $q_k\in\cP_3$ such that
    \[
        q_k\in\mathrm{argmin}\left\{\int_0^1\|\tilde w_k(s)-q\|^2_{L^2(\gamma_n)}ds,\quad q\in\cP_3\right\}.
    \]
    Since $\cP_3$ is a vector space, $q_k$ satisfies
    \begin{equation}\label{eq:orthogonality conditions}
        \int_0^1 \int_{\R^n}q(\tilde w_k(s)-q_k)d\gamma_nds =0\quad\forall q\in\cP_3.
    \end{equation}

    \medskip\noindent\textit{Step 1.} For all $R>0$ there is $C(R)$ such that, for $k$ large enough,
    \begin{equation}\label{eq:L2 Linfty}
        |\tilde w_k - p_k|^2\le C(R)\left(\int_0^1\|\tilde w_k-p_k\|^2_{L^2(\gamma_n)}ds + e^{-\sigma_k}\right)\quad\text{in }B_R\times(R^{-2},1).
    \end{equation}
    Note that if $k$ is large enough then $\zeta_{\sigma_k}\equiv1$ on $B_{2R}\times(0,1)$, hence $\tilde w_k =\tilde v_k(\cdot,\cdot+\sigma_k)$ and
    \[
        (\cL_3-\partial_s) \tilde w_k = -e^{(s+\sigma_k)/2}\chi_{\{\tilde u(\cdot,\cdot+\sigma_k)=0\}}\quad\text{on }B_{2R}\times(0,1),\quad k\gg1.
    \]
    Since $p_k\in\cP_3^+$ then $(\cL_3-\partial_s)p_k=\cL_3p_k\le0$, thus
    \[
        (\cL_3-\partial_s) (\tilde w_k-p_k)\ge(\cL_3-\partial_s) \tilde w_k = -e^{(s+\sigma_k)/2}\chi_{\{\tilde u(\cdot,\cdot+\sigma_k)=0\}}\quad\text{on }B_{2R}\times(0,1),\quad k\gg1.
    \]
    Since for all $p\in\cP_3$ (as they are divisible by $x_n$)
    \begin{equation}\label{eq:polynomial lower bound}
        r^2p_2 + r^3p = r^2\frac12\left(x_n + r\frac{p}{x_n}\right)^2-\frac{r^4}2\frac{p^2}{x_n^2}\ge -C r^4.
    \end{equation}
    This implies
    \[
        \tilde u(\cdot,\cdot+\sigma_k) > \tilde p_2+p+Ce^{-\sigma_k/2}\ge0\quad\text{on }\{\tilde v_k-p-Ce^{-\sigma_k/2}>0\},
    \]
    hence
    \[
        (\cL_3-\partial_s) (\tilde w_k-p_k- C e^{-\sigma_k/2})_+ \ge0\quad\text{on }B_{2R}\times(0,1),\quad k\gg1.
    \]
    Thus, Lemma \ref{lem:parabolic ABP} (together with a standard covering argument) yields
    \begin{equation}\label{eq:partial L2 Linfty}
        (\tilde w_k-p_k)_+^2 \le C(R)\left(\int_0^1\|\tilde w_k-p\|^2_{L^2(\gamma_n)}ds+e^{-\sigma_k}\right)\quad\text{on }B_{R}\times(R^{-2},1),\quad k\gg1.
    \end{equation}
    Similarly, we have
    \[
        (\cL_3-\partial_s) (p_k-\tilde w_k) = (\cL_3-\partial_s) p_k + e^{\sigma_k/2}\chi_{\{\tilde u(\cdot,\cdot+\sigma)=0\}}\ge (\cL_3-\partial_s) p_k\quad\text{on }B_{2R}\times(0,1)
    \]
    for $k\gg1$.
    Since $\cL_3 p_k=0$ in $\R^n\setminus\{y_n=0\}$ and $p_k-\tilde w_k =-\tilde u_k\le0$ on $\{y_n=0\}$, this implies
    \[
        (\cL_3-\partial_s) (p_k-\tilde w_k)_+\ge0\quad\text{on }B_{2R}\times(0,1),\quad k\gg1,
    \]
    hence Lemma~\ref{lem:parabolic ABP} together with a covering argument yields
    \[
        (p_k-\tilde w_k)_+^2\le C \int_0^1\|p_k-\tilde w_k\|^2_{L^2(\gamma_n)}ds\quad\text{on }B_{R}\times(R^{-2},1),\quad k\gg1.
    \]
    Combining this with \eqref{eq:partial L2 Linfty} we find \eqref{eq:L2 Linfty}.

    \medskip\noindent\textit{Step 2.} There is $C>0$ independent from $k$ so that for $k$ large enough
    \begin{equation}\label{eq:almost homogeneity}
        \int_0^1\|\partial_s\tilde w_k\|^2_{L^2(\gamma_n)}ds+ \frac1{10}e^{-\sigma_k/2} \le \frac Ck\left(\int_0^1\|\tilde w_k-q_k\|^2_{L^2(\gamma_n)}ds\right).
    \end{equation}
    Note that \eqref{eq:conformal monotonicity weiss} and \eqref{eq:constant weiss} yield
    \begin{equation}\label{eq:exploit constancy weiss}
        \int_0^1 \|\partial_s\tilde w_k\|^2_{L^2(\gamma_n)}ds + \frac19e^{-\sigma_k/2}\le \frac1k\int_0^1W_3(\tilde w_k,s)ds
    \end{equation}
    for $k$ large enough.
    We now bound $W_3(\tilde w_k,s)$. To show this bound, we omit the dependence on $k$ to simplify the notation. Since $W_3(q)=0$ for all $q\in\cP_3$, we compute
    \begin{equation*}\begin{split}
        W_3(\tilde w,s) = W_3(\tilde w,s)-W_3(q)
        &= -W_3(q-\tilde w,s)- 2\int_{\R^n} (\nabla \tilde w\cdot \nabla(q-\tilde w)-\tfrac32 \tilde w(q-\tilde w))d\gamma_n\\
        &\le \tfrac32 \|\tilde w-q\|^2_{L^2(\gamma_n)} + 2\int_{\R^n} \cL_3\tilde w (q-\tilde w) d\gamma_n,
    \end{split}\end{equation*}
    where we used $W_3(q-\tilde w,s) \ge -\tfrac32\|\tilde w(s)-q\|^2_{L^2(\gamma_n)}$ and \eqref{eq:integration by parts}.
    Since $\tilde v(\cdot,\cdot+\sigma)$ solves \eqref{eq:conformal Stefan} and using \eqref{eq:exponential errors} we find
    \begin{multline}\label{eq:proof bound weiss 1}
        W_3(\tilde w,s)\le \tfrac32 \|\tilde w(s)-q\|^2_{L^2(\gamma_n)} + 2\int_{\R^n\times\{s\}}\partial_s\tilde w(q-\tilde w)d\gamma_n\\
        -2\int_{\R^n\times\{s\}}(q-\tilde w)\chi_{\{\tilde u(\cdot,\cdot+\sigma)=0\}}e^{(s+\sigma)/2}d\gamma_n+C\exp(-e^{\sigma/2}),
    \end{multline}
    for some constant $C>0$ depending only on $n,\zeta,M$.
    We now note that $\tilde w\chi_{\{\tilde u(\cdot,\cdot+\sigma)=0\}}=-\zeta_\sigma\tilde p_2 \chi_{\{\tilde u(\cdot,\cdot+\sigma)=0\}}\le0$. Moreover, Lemma~\ref{lem:closeness contactSet} yields $\{\tilde u(\cdot,\cdot+\sigma)=0\}\subset \{|y_n|\le Ce^{-\sigma/2}\}$ for some $C>0$ depending only on $n,M$. Finally, since $q$ is a polynomial vanishing on $\{y_n=0\}$, it must satisfy $q\ge -C|y_n|$ for some constant $C>0$ depending only on $\|q\|$, which in turn depends only on $n,M$.
    Thus, this yields
    \begin{equation}\label{eq:proof bound weiss 2}
        -2\int_{\R^n\times\{s\}}(q-\tilde w)\chi_{\{\tilde u(\cdot,\cdot+\sigma)=0\}}e^{(s+\sigma)/2}d\gamma_n\le Ce^{-\sigma/2}\quad \forall s\in(0,1)
    \end{equation}
    for some $C>0$ depending only on $n,M$.
    Moreover, by Hölder's inequality
    \begin{equation}\label{eq:proof bound weiss 3}
        2\int_{\R^n\times\{s\}}\partial_s\tilde w(q-\tilde w)d\gamma_n\le \|q-\tilde w(s)\|^2_{L^2(\gamma_n)} + \|\partial_s\tilde w(s)\|^2_{L^2(\gamma_n)}.
    \end{equation}
    Thus, \eqref{eq:proof bound weiss 1} together with \eqref{eq:proof bound weiss 2}, \eqref{eq:proof bound weiss 3} implies
    \[
        \int_0^1 W_3(\tilde w_k,s)ds \le C\int_0^1\|\tilde w_k-q_k\|^2_{L^2(\gamma_n)}ds + \int_0^1\|\partial_s\tilde w_k\|^2_{L^2(\gamma_n)}ds  + Ce^{-\sigma_k/2}.
    \]
    Using this to estimate the right hand side in \eqref{eq:exploit constancy weiss} we find
    \[
        (1-\tfrac1k)\int_0^1\|\partial_s\tilde w_k\|^2_{L^2(\gamma_n)}ds + (\tfrac19-\tfrac Ck)e^{-\sigma_k/2} \le \frac Ck\int_0^1\|\tilde w_k-q_k\|^2_{L^2(\gamma_n)}ds
    \]
    for some $C>0$ independent from $k$. 
    Choosing $k$ large enough \eqref{eq:almost homogeneity} follows.

    \medskip\noindent\textit{Step 3.} There is $C>0$ independent from $k$ such that
    \begin{equation}\label{eq:caccioppoli}
        \int_0^1\|\nabla (\tilde w_k-q_k)\|^2_{L^2(\gamma_n)}ds \le C\left(\int_0^1\|\tilde w_k-q_k\|^2_{L^2(\gamma_n)}ds+e^{-\sigma_k/2}\right)
    \end{equation}
    We omit the dependence on $k$ to simplify the notation.
    We first note that for all functions $w$ sufficiently regular and all $q\in\cP_3$ it holds
    \begin{equation}\label{eq:weiss decomposition}
        W_3(\tilde w-q,s) = W_3(\tilde w,s) + 2\int_{\R^n\times\{s\}}\tilde w\cL_3qd\gamma_n,
    \end{equation}
    since we can compute
    \begin{equation*}\begin{split}
        W_3(\tilde w,s) = W_3(\tilde w-q,s)+W_3(q) &+ 2\int_{\R^n\times\{s\}}(\nabla q\cdot\nabla(\tilde w-q)-\tfrac32 q(\tilde w-q))d\gamma_n\\
        &=W_3(\tilde w-q,s) -2\int_{\R^n\times\{s\}}\tilde w\cL_3qd\gamma_n,
    \end{split}\end{equation*}
    where we used \eqref{eq:integration by parts} and the fact that $W_3(q)=0$ and $q\cL_3q\equiv0$ for all $q\in\cP_3$.
    By rearrangement \eqref{eq:weiss decomposition} follows.\\
    Note also that \eqref{eq:integration by parts} yields, for all sufficiently regular functions $\tilde w$,
    \begin{equation}\label{eq:proof caccioppoli 1}\begin{split}
        W_3(\tilde w,s) = -\int_{\R^n\times\{s\}}\tilde w\cL_3\tilde wd\gamma_n&=-\int_{\R^n\times\{s\}}q\cL_3\tilde wd\gamma_n - \int_{\R^n\times\{s\}} (\tilde w-q)\cL_3\tilde wd\gamma_n\\
        &= -\int_{\R^n\times\{s\}}\tilde w\cL_3qd\gamma_n- \int_{\R^n\times\{s\}} (\tilde w-q)\cL_3\tilde wd\gamma_n.
    \end{split}\end{equation}
    Using now \eqref{eq:conformal Stefan} and \eqref{eq:exponential errors} we have
    \begin{multline}\label{eq:proof caccioppoli 2}
        -\int_{\R^n\times\{s\}}(\tilde w-q)\cL_3\tilde w d\gamma_n\le -\int_{\R^n\times\{s\}}(\tilde w-q)\partial_s\tilde wd\gamma_n\\
        + \int_{\R^n\times\{s\}} (\tilde w-q)e^{(s+\sigma)/2}\chi_{\{\tilde u(\cdot,\cdot+\sigma)=0\}}d\gamma_n+C\exp(-e^{\sigma/2})
    \end{multline}
    We now estimate the second term in \eqref{eq:proof caccioppoli 2} similarly to \eqref{eq:proof bound weiss 2}. Indeed,
    by \eqref{eq:polynomial lower bound} we have
    \[
        \tilde w -q=-e^{(s+\sigma)/2}y_n^2/2-q\le Ce^{-(s+\sigma)/2},
    \]
    hence
    \[
        (\tilde w-q)e^{(s+\sigma)/2}\chi_{\{\tilde u(\cdot,\cdot+\sigma)=0\}}\le C\chi_{\{\tilde u(\cdot,\cdot+\sigma)=0\}}.
    \]
    Since $\{\tilde u(\cdot,\cdot+\sigma)=0\}\subset\{|y_n|\le Ce^{-\sigma/2}\}$ by Lemma~\ref{lem:closeness contactSet}, we compute
    \begin{equation}\label{eq:proof caccioppoli 3}
        \int_{\R^n\times\{s\}}(\tilde w-q)e^{(s+\sigma)/2}\chi_{\{\tilde u(\cdot,\cdot+\sigma)=0\}}d\gamma_n\le C e^{-\sigma/2}.
    \end{equation}
    We can estimate the first term in \eqref{eq:proof caccioppoli 2} by using \eqref{eq:exploit constancy weiss} and computing
    \begin{multline}\label{eq:proof caccioppoli 4}
        \int_0^1\int_{\R^n}(\tilde w-q)\partial_s\tilde wd\gamma_nds\le\left(\int_0^1\|\tilde w-q\|^2_{L^2(\gamma_n)}ds\right)^{1/2}\left(\int_0^1\|\partial_s\tilde w\|^2_{L^2(\gamma_n)}ds\right)^{1/2}\\
        \le \left(\frac1k\int_0^1\|\tilde w-q\|^2_{L^2(\gamma_n)}\right)^{1/2}\left(\int_0^1W_3(\tilde w,s)ds\right)^{1/2}\\\le \frac1{2k}\int_0^1\|\tilde w-q\|^2_{L^2(\gamma_n)}ds + \frac12\int_0^1W_3(\tilde w,s)ds.
    \end{multline}
    Thus, integrating \eqref{eq:proof caccioppoli 1} in $s\in(0,1)$ and using \eqref{eq:proof caccioppoli 2} together with \eqref{eq:proof caccioppoli 3}, \eqref{eq:proof caccioppoli 4} we find
    \[
        \frac12\int_0^1W_3(\tilde w,s)ds + \int_0^1\int_{\R^n}\tilde w\cL_3qd\gamma_nds \le \frac1{2k}\int_0^1\|\tilde w-q\|^2_{L^2(\gamma_n)}ds+Ce^{-\sigma/2}.
    \]
    Now \eqref{eq:weiss decomposition} yields
    \[
        \frac12\int_0^1 W_3(\tilde w-q,s)ds\le \frac1{2k}\int_0^1\|\tilde w-q\|^2_{L^2(\gamma_n)}ds+Ce^{-\sigma/2},
    \]
    thus
    \[
        \int_0^1\|\nabla(\tilde w_k-q_k)\|^2_{L^2(\gamma_n)}ds\le \left(\frac32+\frac1k\right)\int_0^1\|\tilde w_k-q_k\|^2_{L^2(\gamma_n)}ds+ Ce^{-\sigma_k/2},
    \]
    as we wanted.

    \medskip\noindent\textit{Step 4.} We conclude. Set
    \[
        \delta^2_k\coloneqq \int_0^1\|\tilde w_k-q_k\|^2_{L^2(\gamma_n)}ds,\quad f_k\coloneqq \frac{\tilde w_k-q_k}{\delta_k}.
    \]
    Note that \eqref{eq:almost homogeneity} implies $e^{-\sigma_k/2}=o(\delta_k^2)$. Thus, by \eqref{eq:caccioppoli} and \eqref{eq:almost homogeneity} there is a constant $C>0$ independent from $k$ such that
    \[
        \|f_k\|_{L^2((0,1);H^1(\gamma_n))} + \|\partial_sf_k\|_{L^2((0,1);L^2(\gamma_n))}\le C.
    \]
    By Lemma~\ref{lem:cpt} there is $f_\infty\in L^2((0,1);L^2(\gamma_n))$ such that up to a subsequence
    \[
        f_k\to f_\infty\quad\text{strongly in }L^2((0,1);L^2(\gamma_n)).
    \]
    Note that by \eqref{eq:orthogonality conditions} and strong convergence in $L^2((0,1);L^2(\gamma_n))$ we have
    \begin{equation}\label{eq:contradiction}
        \int_0^1\int_{\R^n}qf_\infty d\gamma_nds =0\quad\forall q\in\cP_3 \quad\text{and}\quad\int_0^1\|f_\infty\|^2_{L^2(\gamma_n)}ds=1.
    \end{equation}
    \eqref{eq:almost homogeneity} yields $\partial_sf_k\to0$ as $k\to\infty$, thus
    \[
        \partial_sf_\infty\equiv0\quad\text{in }\R^n\times(0,1)
    \]
    and, since
    \[
        (\cL_3-\partial_s) f_k = 0\quad\text{in }B_{e^{\sigma_k/2}/10}\times(0,1)\setminus \{|y_n|>Ce^{-\sigma_k/2}\},
    \]
    letting $k\to+\infty$ and using that $\partial_sf_\infty\equiv0$ we find
    \[
        \cL_3f_\infty=0\quad\text{on }\R^n\setminus\{y_n\neq0\}.
    \]
    Moreover, recalling the relation between $u$ and $\tilde w_k$, since $\partial_tp_3\ge c'|y_n|$ and $\int_0^1\|\tilde w_k-p_k\|^2ds\to0$, for all $\delta,R_0>0$ \eqref{eq:L2 Linfty} yields, provided $k$ is large enough,
    \[
        \frac{(u_k-p_2)_{r_k}}{r_k^3} \le -c'|x_n| + C_0|x_n|^3+ x_nQ_2+\delta\quad\text{in }\{-4\le t\le-1,|x|\le 2R_0\}.
    \]
    Thus Lemma \ref{lem:barrier} yields
    \[
        \tilde w_k \le Ce^{-\sigma_k/2}\quad\text{in }B_R\times [R^{-2},1]\cap\{y_n=0\}.
    \]
    for $k$ large enough. Since $q\equiv0$ on $\{y_n=0\}$ and $e^{-\sigma_k/2}=o(\delta_k^2)$, this yields
    \[
        0\le f_k=\frac{\tilde w_k}{\delta_k} \le C\frac{e^{-\sigma_k/2}}{\delta_k}=o(1)\quad\text{in }B_R\times [R^{-2},1]\cap\{y_n=0\}.
    \]
    Letting $k\to\infty$ we find that $f_\infty$ solves
    \[\begin{cases}
        \cL_3f_\infty=0     &\text{in }\R^n\setminus\{y_n=0\},\\
        f_\infty\equiv0     &\text{in }\{y_n=0\}.
    \end{cases}\]
    As a consequence $f_\infty\in\cP_3$, contradicting \eqref{eq:contradiction}.
\end{proof}

\subsection{Proof of $C^{3+\beta}$ expansion}
Before proving the main result of this section we recall few useful facts.
Given $f\colon\R^n\times(-1,0)$ and the reverse heat kernel $G_n$ (see \eqref{eq:gaussian kernels}), we define the parabolic frequency as
\[
    \phi(r,f) = \frac{D(r,f)}{H(r,f)}
\]
where
\[
    D(r,f) = r^2\int_{\R^n}|\nabla f(\cdot,-r^2)|^2G(\cdot,-r^2)dx,\quad H(r,f) = \int_{\R^n}f^2(\cdot,-r^2)G(\cdot,-r^2)dx.
\]
\begin{lem}[{\cite[Corollary 8.5]{FRS}}]\label{lem:uniform cubic decay}
    Let $u\colon B_1\times[-1,1]\to[0,M]$ solve \eqref{eq:SP}, satisfying \eqref{eq:condition nondegeneracy partial_t} with $(0,0)\in\Sigma_{n-1}$ and first blow-up $p_2$.
    Then there is $C>0$ depending only on $c,M$ such that
    \[
        C^{-1}r^6\le H(\zeta(u-p_2),r)\le Cr^6\quad \forall r\in(0,1/2).
    \]
\end{lem}
\begin{proof}
    It follows from the proof of \cite[Corollary 8.5]{FRS}, noting that the constant depends only on $M,c$.
\end{proof}
\begin{lem}[{\cite[Proposition 5.4]{FRS}}]
    Let $u\colon B_1\times[-1,1]\to[0,M]$ solve \eqref{eq:SP}, satisfying \eqref{eq:condition nondegeneracy partial_t} with $(0,0)\in\Sigma_{n-1}$ and first blow-up $p_2$.
    Then
\begin{equation}\label{eq:frequency almost monotonicity} 
    \frac d{dr}\phi(\zeta(u-p_2),r)\ge -C\exp(-1/2r)
\end{equation}
for some $C>0$ depending only on $n$, $\|u\|_{L^\infty(C_1)}$ and $c$.
\end{lem}
\begin{proof}
    The proof is the same as \cite[Proposition 5.4]{FRS}, simply noting that in our context Lemma~\ref{lem:uniform cubic decay} yields $-Ce^{-2/r}/H(r,w)\ge - Ce^{-1/r}$, for $C>0$ depending only on $M,c$, and letting $\gamma\to+\infty$.
\end{proof}
We also point out the following consequence of Theorem~\ref{teo:frequency gap}.
\begin{cor}\label{cor:frequency gap}
        Let $u\colon B_1\times[-1,1]\to [0,+\infty)$ be a bounded solution of the Stefan problem \eqref{eq:SP} such that $(0,0)\in\Sigma_{n-1}$. Then $\lim_{r\to0^+}\phi(\zeta(u-p_2),r)=3$.
    \end{cor}
    \begin{proof}
        It is an immediate consequence of Theorem~\ref{teo:frequency gap} together with \cite[Lemma 5.8 (b) and Proposition 6.7 (b)]{FRS}.
    \end{proof}
\begin{proof}[Proof of Theorem~\ref{teo:C3beta expansion}]
    We split the proof in several steps. In Step 1 we show that we can apply the epiperimetric inequality at all scale  $0<r<\bar r$ for some $\bar r$ depending only on $M,c,\rho$.
    In Step 2 we apply Proposition~\ref{prop:epiperimetric stefan}, working in conformal coordinates. In Step 4, using an $L^2-L^\infty$ estimate from Step 3, we conclude.

    Given $M,c,\rho>0$, let $\delta_0$ from Proposition~\ref{prop:epiperimetric stefan} applied with $M,c'$, where $c'$ depends only on $n,c$ and will be set in Step 1.

    Given a solution $u$ of \eqref{eq:SP} we set
    \[
        v= u-\tfrac12x_n^2.
    \]
    Note also that, in $(x,t)$ coordinates, the modified Weiss energy defined in \eqref{eq:modified weiss} is
    \[
        \widetilde W(\zeta v,r) = W_3(\zeta v,r) + r,
    \]
    where $W_3$ is defined in \eqref{eq:weiss} and $\zeta$ in \eqref{eq:spatial cutoff}.

    \medskip\noindent\textit{Step 1.} There is $\bar r$ (depending only on $M,c,\rho$) such that for all $u\in\cS(M,c,\rho)$, all $(x_0,t_0)\in \Sigma_{n-1}(u)\cap B_{1-\rho}\times[-1+\rho^2,1]$ and all $r<\bar r$ there are $p_r\in\cP_3^+$ satisfying $\partial_tp_r\ge c|x_n|$ and such that, up to a rotation in space,
    \[
        \|(\zeta v-p_r)(x_0+r\cdot,t_0+r^2\cdot)\|_{L^2((-1,-1/4);L^2(\gamma_n))}\le \delta_0 r^3\quad\text{and}\quad \widetilde W(\zeta v,r)<1\quad\forall r<\bar r.
    \]
    Assume by contradiction that the claim is false. Then there are solutions $u_k\in \cS(M,c,\rho)$, singular points $(x_k,t_k)$ such that
    \[
        (x_k,t_k)\in\Sigma_{n-1}(u_k)\cap B_{1-\rho}\times[-1+\rho^2,1]
    \]
    but there are $r_k\to0$ such that
    \[
        \|(\zeta v_k -p)(x_k+r_k\cdot,t_k+r_k^2\cdot)\|_{L^2((-1,-1/4);L^2(\gamma_n))} >\delta_0 r_k^3
    \]
    for all $p\in\cP_3^+$ satisfying $\partial_tp^\even\ge c'|x_n|$, or $\widetilde W(w_k,r_k)\ge1$.
    Up to a rescaling and translation, we will assume $u_k\colon B_1\times[-1,0]\to[0,M']$ are equibounded solutions of \eqref{eq:SP} and $x_k=t_k=0$. Thus, by local a priori estimates \eqref{eq:local C11 estimates} there is $u_\infty$ such that
    \[
        u_k\to u_\infty\quad\text{in }C^{1,1}_{x,\loc}(C_1)\cap C^{0,1}_{t,\loc}(C_1).
    \]
    We note that (see \cite{Blanchet06,LindgrenMonneau15,CaffarelliPetrosyanShahgolian}) there is modulus of continuity $\omega(r)$ depending only on $n,M$ such that
    \[
        v_k(r\cdot,r^2\cdot) = \omega(r)r^2\quad \forall r\in(0,1).
    \]
    Letting $k\to +\infty$ this implies $(0,0)\in\Sigma_{n-1}(u_\infty)$.
    As a consequence of Corollary~\ref{cor:frequency gap} it holds $\lim_{r\to0^+}\phi(\zeta v_k,r)=\lim_{r\to0^+}\phi(\zeta v_\infty,r)=3$.
    This together with \eqref{eq:frequency almost monotonicity} yields
    \[
        3\le \phi(\zeta v_\infty,r)+C\exp(-1/r)\le 3 + \sigma(r), 
    \]
    where $\sigma(r)\to0$ as $r\to0$.
    The same argument applied to the functions
    \[
        w_k\coloneqq r_k^{-3} (\zeta v_k)(r_k\cdot,r_k^2\cdot).
    \]
    together with $C^{1,1}$ convergence of $u_k$ to $u_\infty$ yields for all $R,\delta>0$
    \begin{equation}\label{eq:frequency pinching}
        3-\delta\le\phi(\zeta v_k,r)\le3+\delta\quad\forall r\in(0,R),\quad k\gg1.
    \end{equation}
    Similarly, since $W_3(\zeta(u_\infty-x_n^2/2),r)\to0$ as $r\to0$, $C^{1,1}$ convergence implies
    \[
        \widetilde W(\tilde w,r_k)\le \tfrac12
    \]
    provided $k$ is large enough.
    Moreover, \cite[Corollary 6.2, Lemma 6.5]{FRS} and Lemma~\ref{lem:uniform cubic decay} imply that for all $R>0$ there is $C(R)>0$ such that
    \begin{equation}\label{eq:a priori estimates}
        \|w_k\|_{L^\infty(C_R)}+\|\nabla w_k\|_{L^\infty(C_R)} + \|\partial_t w_k\|_{L^\infty(C_R)}\le C(R).
    \end{equation}
    In addition, since $\partial_tx_n^2/2\equiv0$, the nondegeneracy condition~\eqref{eq:condition nondegeneracy partial_t} implies
    \begin{equation}\label{eq:uniform nondegeneracy dt}
        \fint_{C_1} \partial_t w_k \ge c\quad \forall k\ge1.
    \end{equation}
    Finally, since the solutions $u_k$ are uniformly bounded, the cubic scaling defining $w_k$ yields uniform polynomial growth at $\infty$, namely there is $C>0$ depending only on $n,\zeta,M,\rho$ such that
    \begin{equation}\label{eq:cubic growth}
        |w_k|\le C R^3\quad \text{in }C_R\quad\forall R\ge1\forall k\ge1.
    \end{equation}

    \smallskip\noindent We now note that thanks to \eqref{eq:a priori estimates} there is $q\in H^1_{\loc}(\R^n\times(-\infty,0))$ such that
    \[
        w_k \to q\quad \text{locally weakly in } H^1_\loc (\R^n\times(-\infty,0)).
    \]
    We now claim that $q(\cdot,-1)\in\cP_3^+$ satisfies $\partial_tq\ge c'|x_n|$.
    Indeed, by Lemma~\ref{lem:closeness contactSet} the functions $w_k$ solve for all $R>0$ and $k$ large enough
    \[
        \Heat w_k = 0\quad \text{in }B_R\times(-R^2,0)\setminus\{|x_n|\le Cr_k\},
    \]
    thus $\Heat q=0$ on $\R^n\times(-\infty,0)\setminus\{x_n=0\}$.
    Moreover, since $\Heat w_k=-r_k^{-1}\chi_{\{u_k(r_k\cdot,r_k^2\cdot)=0\}}\le0$ are nonnegative measures and they weakly converge to $\Heat q$, we also have $\Heat q\le0$.
    By Lipschitz estimates we also have $w_k\to q$ locally uniformly in $\R^n\times(-\infty,0)$.
    Since $w_k = u_k(r_k\cdot,r_k^2\cdot)\ge0$ on $\{x_n=0\}$, this implies $q\ge0$ on $\{x_n=0\}$.
    Thus $q\Heat q\le0$. However, the nonnegative measures $w_k\Heat w_k = \tfrac12 x_n^2\chi_{\{u_k(r_k\cdot,r_k^2\cdot)=0\}}\ge0$ converge to $q\Heat q$, thus $q\Heat q\equiv0$. It follows that $q$ solves the parabolic thin obstacle problem.
    We also note that the estimate \eqref{eq:cubic growth} yields
    \[
        |q|\le CR^3\quad\text{in }B_R\times(-R^2,0).
    \]
    Moreover, \eqref{eq:frequency pinching} together with \cite[Lemma 5.6 (b)]{FRS} yields
    \[
        \int_{\R^n}q(\cdot,-r^2) G_n(\cdot,-r^2)dx \ge C_\delta r^{6+3\delta}.
    \]
    Since $\delta>0$ is arbitrary, these growth estimates imply that $q$ is 3-homogeneous, i.e. $q\in\cP_3^+$.\\
    Finally, \eqref{eq:uniform nondegeneracy dt} and the explicit form of $q$ (see \eqref{eq:cubic blow-ups}) yields that there is $c'$ depending only on $c$ and $n$ so that
    \[
        \partial_t q^\even\ge c'|x_n|,
    \]
    thus showing the claim.
    To conclude, we note that \eqref{eq:cubic growth} yields that for all $\eps>0$ there is $R_\eps>0$ independent from $k$ so that
    \[
        \int_{-1}^{-1/4}\int_{\R^n\setminus B_{R_\eps}} \tilde w_k^2d\gamma_n<\eps.
    \]
    This, together with local convergence in $H^1_\loc(\R^n\times(-\infty,0))$, is enough to reach a contradiction.

    \medskip\noindent\textit{Step 2.} There are $\bar r>0,\beta\in(0,\frac12), C>0$ depending only on $M,c,\rho$ such that the following holds:

    Let $u\in \cS(M,c,\rho)$ and let $(x_0,t_0)\in \Sigma_{n-1}(u)\cap C_{1-\rho}$. Then, up to a rotation in space,
    \begin{equation}\label{eq:C3beta in paraboloids}
        \int_{\R^n} (\zeta(u(x_0+\cdot,t_0-r^2)-x_n^2/2)-p_3(\cdot,-r^2))^2G_n(x,-r^2)dx \le Cr^{6+2\beta}\quad\forall r<\bar r
    \end{equation}
    for some $p_3\in\cP_3^+$.
    The proof is a standard consequence of the epiperimetric inequality together with a diadic argument in conformal coordinates.
    We recall that conformal coordinates $(y,s)$ are defined in \eqref{eq:conformal coordinates}. If we set $v=u-p_2$ and we define $\bar\sigma$ so that $e^{-\bar\sigma/2}=\bar r$ then, using the notation \eqref{eq:conformal transformation}, the function $\tilde v$ will solve \eqref{eq:conformal Stefan} in $B_{e^{s/2}}\times[0,+\infty)$. Moreover, by Step 1 the function $\tilde w = \widetilde{\zeta v}$ will satisfy
    \[
        \int_{\bar \sigma}^{\bar\sigma+1}\|\tilde w- p_s\|^2_{L^2(\gamma_n)}ds < \delta_0^2\quad\forall s>\bar \sigma
    \]
    for some $p_s\in \cP_3^+$ satisfying $\partial_t p_s^\even \ge c|x_n|$. Up to taking $\bar r$ smaller, we can assume that $\bar \sigma>s_0$. Thus, we can apply Proposition~\ref{prop:epiperimetric stefan} at all times $s\ge\bar\sigma$. Setting $s_k\coloneqq \bar\sigma+k$, applying Proposition~\ref{prop:epiperimetric stefan} yields
    \[
        \widetilde W(\tilde w,s_k) \le e^{-ck}\widetilde W(\tilde w,\bar\sigma)\quad\forall k\ge0,
    \]
    for some $c>0$ depending only on $\eps_0$. Note also that \eqref{eq:conformal monotonicity weiss} yields
    \[
        \|\partial_s\tilde w\|^2_{L^2(\gamma_n)}\le -\frac d{ds}\widetilde W(\tilde w,s) +\frac d{ds} e^{-s/2} + C\exp(-e^{s/2})\le -\frac d{ds}\widetilde W(\tilde w,s)
    \]
    provided $s>\bar\sigma$ is chosen possibly larger, depending only on $n,\zeta,M$.
    Thus, for all $h,k>0$ we compute
    \begin{multline*}
        \|\tilde w(\cdot,s_{k+h})-\tilde w(\cdot,s_k)\|_{L^2(\gamma_n)}\le \sum_{j=0}^{h-1} \|\tilde w(\cdot,s_{k+j+1})-\tilde w(\cdot,s_{k+j})\|_{L^2(\gamma_n)}\\
        \le \sum_{j=0}^{h-1} \left(\int_{s_{k+j}}^{s_{k+j+1}} \|\partial_s\tilde w\|^2_{L^2(\gamma_n)}ds\right)^{1/2}
        \le \sum_{j=0}^{h-1}\left(\widetilde W(\tilde w,s_{k+j})-\widetilde W(\tilde w,s_{k+j+1})\right)^{1/2}\\
        \le \sum_{j=0}^{h-1} e^{-c(k+j)/2}(\widetilde W(\tilde w,\bar\sigma))^{1/2}
        \le Ce^{-ck/2}.
    \end{multline*}
    It follows that the sequence $\tilde w(\cdot,s_k)$ is Cauchy. Since any accumulation point of $\tilde w(\cdot,s_k)$ is in $\cP_3^+$, there is $p_3\in \cP_3^+$ such that $\tilde w(\cdot,s_k)\to p_3$ and
    \[
        \|\tilde w(\cdot,\bar\sigma+k)-p_3\|_{L^2(\gamma_n)} \le Ce^{-ck/2}\quad \forall k\ge0.
    \]
    The same computation as before also yields $\|\tilde w(\cdot,s_k)-\tilde w(\cdot,s_k+s)\|_{L^2(\gamma_n)}\le e^{-ck/2}$ for all $s\in(0,1)$, thus
    \[
        \|\tilde w(\cdot,s)-p_3\|_{L^2(\gamma_n)}\le Ce^{-c(s-\bar\sigma)/2}\quad\forall s\ge\bar\sigma.
    \]
    Recalling \eqref{eq:conformal coordinates} and \eqref{eq:conformal transformation}, a change of variables yields \eqref{eq:C3beta in paraboloids}.

    \medskip\noindent\textit{Step 3.} There is $C>0$ depending only on $n,M$ such that for all $u$ solving \eqref{eq:SP} in $B_1\times[-1,0]$ with $|u|\le M$ and all $p\in\cP_3^+$ satisfying $\|p\|\le M$
    \begin{equation}\label{eq:local uniform estimates}
        |r^{-3}(u(r\cdot,r^2\cdot)-\tfrac{r^2}2x_n^2)-p|\le C (\|r^{-3}(u(r\cdot,r^2\cdot)-\tfrac{r^2}2x_n^2)-p\|_{L^2(C_1)} + r)\quad\text{in }C_{1/2}.
    \end{equation}
    The proof is similar to \eqref{eq:L2 Linfty}. Setting
    \[
        v_r\coloneqq r^{-3}(u(r\cdot,r^2\cdot)-\tfrac{r^2}2x_n^2),\quad u_r = u(r\cdot,r^2\cdot)
    \]
    then $v_r$ solve
    \[
        \Heat v_r = -\tfrac1r\chi_{\{u_r=0\}}\quad\text{in }C_1.
    \]
    Note that, since $\Heat p$ is supported on $\{x_n=0\}$, we have
    \[
        \Heat (p-v_r) =\tfrac1r\chi_{\{u_r=0\}}\ge0\quad\text{in }C_1\setminus\{x_n=0\}
    \]
    and, since $p\equiv0$ and $v_r=r^{-3}u_r\ge0$ on $\{x_n=0\}$, we also have $p-v_r\le 0$ on $\{x_n=0\}$. Thus,
    \[
        \Heat (p-v_r)_+\ge0\quad\text{in }C_1
    \]
    and Lemma \ref{lem:parabolic ABP} yields
    \begin{equation}\label{eq:proof local uniform 1}
        p-v_r\le C \|p-v_r\|_{L^2(C_1)}\quad\text{in }C_{1/2}.
    \end{equation}
    Similarly, since $-\Heat p\ge0$, we compute
    \[
        \Heat (v_r-p)\ge -\frac1r\chi_{\{u_r=0\}}\quad\text{in }C_1.
    \]
    We now note that, recalling \eqref{eq:polynomial lower bound}, on $\{v_r-p-Cr>0\}$ we have $u_r >\tfrac{r^2}2x_n^2 + r^3 p +Cr^4 \ge0$, thus
    \[
        \Heat (v_r-p-Cr)_+\ge0\quad\text{in }C_{1/2}
    \]
    and Lemma \ref{lem:parabolic ABP} yields
    \[
        v_r-p \le C(\|v_r-p\|_{L^2(C_1)} +r)\quad\text{in }C_{1/2}.
    \]
    Recalling~\eqref{eq:proof local uniform 1}, \eqref{eq:local uniform estimates} follows.
    
    \medskip\noindent\textit{Step 4.} We show that there is $C>0$ such that, up to a rotation in space,
    \[
        \|(u-\tfrac12x_n^2-p_3)(r\cdot,r^2\cdot)\|_{L^2(C_1)}\le Cr^{3+\beta}\quad\forall r<\bar r.
    \]
    This, together with \eqref{eq:local uniform estimates}, concludes the proof. Setting
    \[
        w=u-\tfrac12x_n^2-p,\quad w_r=w(r\cdot,r^2\cdot),
    \]
    we first claim that for all $A\ge1$ there is $C_A\ge1$ (depending only on $A,n,M,\beta$) such that the following implication holds:
    \begin{equation}\label{eq:proof claim}
        \|w_r\|_{L^2(C_1)} \ge C_A r^{3+\beta}\quad\Longrightarrow\quad \|w_{2r}\|_{L^2(C_1)}\ge A\|w_r\|_{L^2(C_1)}.
    \end{equation}
    To show \eqref{eq:proof claim} we will instead assume
    \[
        \|w_{2r}\|_{C^1}\le A\|w_r\|_{C^1}
    \]
    and prove
    \[
        \|w_r\|_{L^2(C_1)}\le C_Ar^{3+\beta}.
    \]
    By \eqref{eq:local uniform estimates} the assumption implies
    \[
        \|w_r\|_{L^\infty(C_1)}\le C'(\|w_r\|_{L^2(C_2)}+r^4)\le C''(\|w_{2r}\|_{L^2(C_1)}+r^4)\le C(A\|w_r\|_{L^2(C_1)}+r^4)
    \]
    for some constant $C$ depending only on $n,M$.
    Thus, for all $\tau>0$ small enough we find
    \[
        \int_{B_1\times(-1,-\tau)}w_r^2 \ge \int_{C_1}w_r^2 - \tau\|w_r\|^2_{L^\infty(C_1)}\ge (1-\tau C^2A^2)\|w_r\|_{L^2(C_1)}^2-\tau C^2r^8.
    \]
    Choosing $\tau = (2C^2A^2)^{-1}$, since $G\ge c_\tau$ on $B_1\times(-1,-\tau)$ and recalling \eqref{eq:C3beta in paraboloids}, we find
    \[
        \tfrac12\|w_r\|^2_{L^2(C_1)} \le \int_{B_1\times(-1,-\tau)}w_r^2 +Cr^8\le C\left(\int_{B_1\times(-1,-\tau)} w_r^2G +r^8\right)\le Cr^{6+2\beta}
    \]
    for some $C$ depending only on $n,M,A,\beta$, as we wanted.\\
    We now conclude the proof. Let $N>1$ be a large constant to be fixed, depending only on $n,M$, and let $A=N2^{3+\beta}$. Assume by contradiction that there is $r<\bar r/2$ such that
    \[
        \|w_r\|_{L^2(C_1)}\ge C_A r^{3+\beta},
    \]
    where $C_A$ is given by \eqref{eq:proof claim}.
    Then \eqref{eq:proof claim} yields
    \[
        \|w_{2r}\|_{L^2(C_1)}\ge A\|w_r\|_{L^2(C_1)}\ge AC_Ar^{3+\beta}\ge AC_A2^{-(3+\beta)}(2r)^{3+\beta}\ge NC_A(2r)^{3+\beta}.
    \]
    Thus, we can still apply \eqref{eq:proof claim} to find
    \[
        \|w_{4r}\|_{L^2(C_1)}\ge M\|w_{2r}\|_{L^2(C_1)}\ge ANC_A (2r)^{3+\beta} = N^2 C_A (4r)^{3+\beta},
    \]
    and iterating $\ell$ times yields
    \[
        \|w_{2^\ell r}\|_{L^2(C_1)}\ge N^\ell C_A (2^\ell r)^{3+\beta}.
    \]
    Choose $\ell\ge1$ so that $\bar r/2<2^\ell r\le\bar r$. Since
    \[
        \|w_{\rho}\|_{L^2(C_1)}\le C\rho^{3+\beta}\quad\forall \rho\in(\bar r/2,\bar r]
    \]
    for some constant $C>0$ depending only on $n,M,\bar r$, we find
    \[
        C (2^\ell r)^{3+\beta}\ge \|w_{2^\ell r}\|_{L^2(C_1)} \ge N(2^\ell r)^{3+\beta},
    \]
    thus reaching a contradiction if $N$ is large enough.
\end{proof}

\section{Proof of Theorem~\ref{teo:smooth covering}}\label{sec:main results}

To show Theorem~\ref{teo:smooth expansion}, we will use the following $C^\infty$ expansion, proven in \cite{FRS}.
We say that $u$ satisfies a $C^{3,\beta}$ expansion at $(0,0)\in\Sigma_{n-1}$ provided
\begin{equation}\label{eq:C3beta expansion}
    |u(rx,r^2t)-r^2x_n^2/2 + r^3p_3|\le C_0r^{3+\beta}\quad\forall r\in(0,1)
\end{equation}
for some $C_0>0$ and $p_3\in \cP_3^+$.
We construct a series of polynomial Ansätze for the Taylor expansion of $u$ at a singular point in the maximal stratum, based on~\cite[Definitions 13.3 and 13.4]{FRS}.
\begin{defn}[Two-sided Ansätze]\label{defn:two sided ansatz}
    Let $k\ge3$, and let $(Q_\ell^\pm)_{2\le\ell\le k-1}$ be two families of parabolically homogeneous polynomials of degree $\ell$ satisfying $\Heat(x_nQ_\ell) \equiv0$.
Then, given $\tau\in\R$ and a rotation $R \in\mathrm{SO}(n)$, we define
\[
    \mathscr P_k = \mathscr P_k [Q_2^\pm,\dots,Q_{k-1}^\pm,\tau,R] (x,t)
\]
by
\[
\mathscr P_k(x,t) := \frac12
\mathscr A_k[Q_2^+,\dots,Q_{k-1}^+]^2_+(R(x+ \tau\mathbf e_n),t)+\frac12\mathscr A_k[Q_2^-,\dots,Q_{k-1}^-]^2_-(R(x+ \tau\mathbf e_n),t),
\]
where $\mathscr A_k[Q_2,\dots,Q_{k-1}]$ is given in~\cite[Definition 13.3]{FRS}.
\end{defn}

\begin{teo}[{\cite[Theorems 13.1 and 13.5]{FRS}}]\label{teo:C3beta implies Cinfty}
    For all $C_0>0,\beta\in(0,1)$, $\alpha\in(0,1), k\ge3$ there is $\bar r>0$ depending only on $\alpha,k,\bar r, C_0,\beta$ such that the following holds:
    
    Let $u$ solve \eqref{eq:SP} in $B_1\times[-1,1]$ with $(0,0)\in \Sigma_{n-1}$ satisfying~\eqref{eq:C3beta expansion}.
    Then there is a two-sided polynomial Ansatz $\mathscr P_k=\mathscr P_k[Q_2^\pm\dots,Q_{k-1}^\pm,0,I]$ (see Definition \ref{defn:two sided ansatz}) such that
    \[
        \|u-\mathscr P_k\|_{L^2(B_r\times(-r^2,-r^{2+\beta/2})}\le r^{k+\alpha}\quad\forall r<\bar r.
    \]
\end{teo}

\begin{proof}[Proof of Theorem~\ref{teo:smooth expansion}]
    Given $M,c,\rho>0$, let $u\in\cS(M,c,\rho)$.
    By Theorem~\ref{teo:C3beta expansion} there are $\bar r, C_0$ and $\beta>0$ such that the following holds: 
    for all $(x_0,t_0)\in\Sigma_{n-1}(u)\cap B_{1-\rho}\times[-1+\rho^2,1]$ there is a rotation $R_{x_0,t_0}$ such that
    \[
        \bar u\coloneqq \bar r^{-2} u(x_0+\bar rR_{x_0,t_0}\cdot,t_0+\bar r^2\cdot)
    \]
    satisfies the $C^{3,\beta}$ expansion \eqref{eq:C3beta expansion} for some $p_3$. Thus, applying Theorem~\ref{teo:C3beta implies Cinfty} the result follows.
\end{proof}

To show Theorem~\ref{teo:smooth covering} we will need the following GMT results.
\begin{lem}[{\cite[Corollary~7.8]{FRS}}]\label{lem:gmt cleaning}
    Let $E\subset\R^n\times[-1,1]$, let $(x,t)$ denote a point in $\R^n\times[-1,1]$, and let $\pi_1 \colon (x,t) \to x$
and $\pi_2 \colon (x,t) \to t$ be the standard projections. Assume that for some $\beta\in(0,n]$ and $s>0$ with $\beta< s$ we have:
\begin{enumerate}[label=\roman*)]
    \item $\dim_\cH\pi_1(E)\le\beta$;
    \item for all $(x_0,t_0) \in E$ and $\eps>0$ there exists $\rho=\rho(x_0,t_0,\eps)>0$ such that
    \[
        \{(x,t) \in B_{\rho}(x_0)\times[-1,1] \colon t-t_0>|x-x_0|^{s-\eps}\} \cap E= \emptyset.
    \]
    \end{enumerate}
    Then $\dim_\cH(\pi_2(E))\le\beta/s$.
\end{lem}
\begin{lem}[{\cite[Theorem 3.8]{Mattila}}]\label{lem:GMT rectifiability}
    Let $E\subset\R^{n+1}$. Assume that for all $x_0\in E$ there are $r,C>0$ and an $(n-2)$-dimensional plane $L\subset \R^n$ such that
    \[
        E\cap B_r(x_0)\subset \{|\pi_L^\perp (x-x_0)|_\pb\le C|\pi_L(x-x_0)|_\pb\},
    \]
    where $\pi_L$ denotes the orthogonal projection onto $L$ and $\pi_L^\perp$ denotes the orthogonal projection onto $L^\perp$.
    Then $E$ can be covered by the images of countably many parabolically Lipschitz functions $f_i\colon\R^{n-2}\to\R^{n+1}$.
\end{lem}

\begin{proof}[Proof of Theorem~\cref{teo:smooth covering}]
    The fact that $\Sigma\setminus\Sigma_{n-1}$ is countably parabolically $(n-2)$-rectifiable follows from Lemma~\ref{lem:parabolic rectifiability lower strata}. Here we show the result with $\Sigma_{n-1}$.

    \medskip\noindent\textit{Step 1.} Let $(x_0,t_0)\in\Sigma_{n-1}$. Then there is $r_0>0$ such that $\Sigma_{n-1}\cap B_{r_0}(x_0) \times [-r_0^2,r_0^2]$ is covered by an $(n-1)$-dimensional $C^\infty$ manifold in $\R^{n+1}$.
    We sketch the proof, as the interested reader can find the details in the proof of \cite[Theorem 1.3]{FRS}.
    By Lemma~\ref{lem:non-degeneracy time derivative single solution} we can apply Theorem~\ref{teo:smooth expansion} to $u$ on compact subsets of $\Omega\times[0,T]$. Setting $K_r\coloneqq B_r\times(-r^2,-r^2/100)$, given $(x_1,t_1)\in\Sigma_{n-1}\cap B_{r_0}(x_0) \times [t_0-r_0^2,t_0+r_0^2]$ this yields
    \begin{equation}\label{eq:expansion}
        \|u(x_1+\cdot,t_1+\cdot)-\mathscr P_k\|_{L^\infty(K_r)}\le C_k r^{k+1/2}
    \end{equation}
    for some two-sided Ansatz $\mathscr P_k$ and some constants $C_k,r_0$ where $r_0,C_k$ are locally independent on the point ($C_k$ might depend on $k$).
    Arguing as in \cite{FRS}, up to a rotation in space this implies the existence of two smooth functions $\mathscr G^{(i)}\colon \R^{n-1}\times\R\to \R$ such that, writing $\R^n\ni x = (x',x_n)\in\R^{n-1}\times\R$,
    \begin{equation}\label{eq:tangential intersection}
        \Sigma_{n-1} \subset \{x_n=\mathscr G^{(1)}(x',t)\}\cap\{x_n=\mathscr G^{(2)}(x',t)\}\cap\{\nabla_{x'} \mathscr G^{(1)}(x',t)=\nabla_{x'}\mathscr G^{(2)}(x',t)\}
    \end{equation}
    in $B_{r_0}(x_0) \times [t_0-r_0^2,t_0+r_0^2]$.
    Using the relation between $\partial_t\mathscr P_3$ and $\partial_t\mathscr G^{(i)}$, condition \eqref{eq:condition nondegeneracy partial_t} yields $\partial_t \mathscr G^{(1)}(x_1,t_1)\neq \partial_t \mathscr G^{(2)}(x_1,t_1)$ at all singular points. This implies the claim.

    \medskip\noindent\textit{Step 2.} We claim that Theorem~\ref{teo:smooth covering} holds with $\Sigma^\infty$ given by
    \[
        \Sigma^\infty \coloneqq\bigcap_{k\ge3}\Sigma_{n-1}^{\ge k},
    \]
    where $\Sigma_{n-1}^{\ge k}$ is defined as follows. Set $\mathscr G^\even(x',t) = \mathscr G^{(1)}(x',t)-\mathscr G^{(2)}(x',t)$ and assume (up to exchanging the indices) that $\partial_t\mathscr G^\even(x',t)<0$.
    Given $k\ge 3$ we define
    \[
        \Sigma_{n-1}^{\ge k}\coloneqq\{(x_0,t_0)\in\Sigma_{n-1}\,:\,\mathscr |\mathscr G^\even(x_0'+x',t_0)|\le C|x'|^{k-1}\}\quad\text{and}\quad
        \Sigma_{n-1}^k=\Sigma_{n-1}^{\ge k}\setminus\Sigma_{n-1}^{\ge k+1}.
    \]

    \smallskip\noindent\textit{Step 2a.} We note that given $(x_0,t_0)\in \Sigma_{n-1}^{\ge k}$ there is $C>0$ such that
    \begin{equation}\label{eq:cleaning}
        t_1\le t_0 + C|x_0-x_1|^{k-1}\quad \forall (x_1,t_1)\in\Sigma_{n-1}.
    \end{equation}
    Indeed, as shown in \eqref{eq:tangential intersection}, if $(x_1,t_1)\in\Sigma_{n-1}$ then
    \[
        0=\mathscr G^\even(x_1,t_1).
    \]
    Since $\mathscr G^\even$ is a smooth function with $\partial_t \mathscr G^\even(x_0,t_0)<0$ and since $(x_0,t_0)\in \Sigma_{n-1}^{\ge k}$ this implies
    \[
        0 = \mathscr G^\even(x_1,t_1)\le C|x_1-x_0|^{k-1} - c(t_1-t_0),
    \]
    for some $C,c>0$, as we wanted.

    \smallskip\noindent\textit{Step 2b.} Let $\pi_t\colon\R^{n+1}\to\R$ denote the projection onto the time axis. Thanks to Step 1 and \eqref{eq:cleaning} we can apply Lemma~\ref{lem:gmt cleaning} with $\beta=n-1$ and $s =k-1$ for all $k\ge3$ to find $\dim_\cH(\pi_t(\Sigma^\infty))=0$.

    \smallskip\noindent\textit{Step 2c.} We conclude by showing that $\Sigma_{n-1}\setminus\Sigma^\infty$ is covered by countably many $(n-2)$-dimensional Lipschitz graphs (with respect to the parabolic structure). Since, as a consequence of $\eqref{eq:tangential intersection}$, $\Sigma_{n-1}=\Sigma^\infty\cup\bigcup_{k\ge3}\Sigma_{n-1}^k$, it is enough to show the result for the sets $\Sigma_{n-1}^k$ for $k\ge3$. Let $L=\{x_{n-1}=x_n=t=0\}$. We claim that for all $k\ge3$ and all $(x_0,t_0)\in\Sigma_{n-1}^k$ there are $C,r>0$ such that, up to a rotation in space,
    \begin{equation}\label{eq:sufficient for rectifiability}\begin{cases}
        |t_1-t_0|\le C|x_1-x_0|^{k-1},\\ |(x_1)_n-(x_0)_n|\le C|x_1-x_0|^2,\\ |(x_1)_{n-1}-(x_0)_{n-1}|\le C|\pi_L(x_1-x_0)|\end{cases}
        \quad\forall (x_1,x_0)\in B_r(x_0)\times (t_0-r,t_0+r).
    \end{equation}
    Indeed, the first inequality follows from~\eqref{eq:cleaning},
    while the second holds since $\mathscr G^{(i)}$ are smooth functions and we can assume, up to a rotation in space, that $\mathscr G^{(i)}(x_0,t_0)=0=\nabla_{x'}\mathscr G^{(i)}(x_0,t_0)$ for $i=1,2$ together with \eqref{eq:cleaning} with $k=3$.
    Moreover, by the definition of $\Sigma_{n-1}^k$ and \eqref{eq:tangential intersection} there is a nonzero $(k-2)$-homogeneous polynomial $g_{k-2}(x')$ such that
    \begin{equation}\label{eq:imposing tangential intersection}
        0=\nabla_{x'}\mathscr G^\even(x_1',t_1) = g_{k-2}(x'_1-x_0') +O(|x_1'-x_0'|^{k-1}+|t_1-t_0|).
    \end{equation}
    Up to a rotation of the first $(n-1)$ coordinates, we can assume that $\{g_{k-2}(x')=0\}\subset\{|x_{n-1}'|\le \frac C4|\pi_L x'|\}$, hence there is $c>0$ such that
    \[
        |g_{k-2}(x')| \ge c(|x'_{n-1}|-\tfrac C2|\pi_L (x')|)_+^{k-2}.
    \]
    Using this in \eqref{eq:imposing tangential intersection} together with $|t_1-t_0|\le C|x_1-x_0|^{k-1}$ and possibly choosing a smaller $r$, \eqref{eq:sufficient for rectifiability} follows. We conclude applying Lemma~\ref{lem:GMT rectifiability}.
\end{proof}

\begin{oss}
    Setting $Q(x_1,x_2) = (x_1^2-x_2^2)^2$ in $\R^3$, by Cauchy-Kovalevskaya Theorem (see for instance \cite[Theorem 4.6.3.2]{Evans}) there is a function $u^{(1)}$ solving~\eqref{eq:SP} in a parabolic neighbourhood of $(0,0)$ satisfying $\{u>0\}=\{x_3>-t+ Q(x_1,x_2)\}$. If $u^{(2)}(x,t) = u^{(1)}(x_1,x_2,-x_3,t)$, then $u=u^{(1)}+u^{(2)}$ will have $\Sigma = \{x_n=t=0, |x_1|=|x_2|\}$ and, using the notations of the proof of Theorem~\ref{teo:smooth covering}, $\Sigma^\infty=\emptyset$.
\end{oss}

\section{Proof of Theorem~\ref{teo:generic R4}}\label{sec:generic}

The proof of Theorem~\ref{teo:generic R4} is based on the following result, whose proof is postponed at the end of the section.
\begin{prop}\label{prop:quadratic cleaning}
    Let $u\colon B_1\times[-1,1]\to[0,+\infty)$ be a bounded solution of \eqref{eq:SP}. Then for all $(x_0,t_0)\in\Sigma\cap B_{1/2}\times[-1/2,1/2]$ there is $C_{x_0,t_0}$ such that
    \begin{equation}\label{eq:quadratic cleaning}
        \Sigma\cap\{|x-x_0|\le r,t\ge t_0+C_{x_0,t_0}r^2\}=\emptyset.
    \end{equation}
\end{prop}
To show Theorem~\ref{teo:generic R4} we will need the following results.
\begin{prop}[{\cite[Theorem 1.9]{LindgrenMonneau15}}]\label{prop:covering lower strata}
    Let $u$ be a solution of \eqref{eq:SP}, then $\cup_{h\le n-2}\Sigma_h$ can be covered by one $(n-2)$-dimensional manifold $C^1$ in space and $C^{1/2}$ in time.
\end{prop}
\begin{lem}\label{lem:parabolic rectifiability lower strata}
    The set $\Sigma\setminus\Sigma_{n-1}$ is countably parabolically $(n-2)$-rectifiable.
\end{lem}
\begin{proof}
    Given a singular point $(x_0,t_0)$ we denote by $C_{x_0,t_0}$ the constant given by Proposition~\ref{prop:quadratic cleaning}.
    Given $N\ge1$ we define
    \[
        \Sigma^N\coloneqq \{(x_0,t_0)\in\Sigma_{\le n-2}\,:\, C_{x_0,t_0}\le N\}.
    \]
    It then follows that if $(x_0,t_0),(x_1,t_1)\in\Sigma^N$ then
    \[
        |t_1-t_0|\le N|x_1-x_0|^2.
    \]
    Moreover, up to a rotation in space we can assume that $\{p_{2,x_0,t_0}=0\}\subset\{x_n=x_{n-1}=0\}=L$.
    Thus, this and Proposition~\ref{prop:covering lower strata} yield
    \[
        |(x_0-x_1)_n|^2+|(x_0-x_1)_{n-1}|^2+|t_0-t_1|\le C|\pi_L(x_0-x_1)|^2,
    \]
    where $\pi_L\colon\R^n\to L$ denotes the orthogonal projection onto $L$.
    Since by Proposition~\ref{prop:quadratic cleaning} $\Sigma_{\le n-2}=\cup_{N\ge1} \Sigma^N$ we conclude using Lemma~\ref{lem:GMT rectifiability}.
\end{proof}
The following fact is essentially \cite[Lemma A.3]{FerryLemma}.

\begin{lem}\label{lem:FerryLemma}
    Let $E\subset \R^n$ and $m\in\N$. Assume that $E$ is countably $m$-rectifiable (with respect to the euclidean structure) and that there is $f\colon E\to \R$ satisfying for some $C>0$ and some $p>1$
    \[
        |f(x)-f(y)|\le C|x-y|^p\quad\forall x,y\in E.
    \]
    Then
    \[
        \cH^\frac mp (f(E))=0.
    \]
\end{lem}
\begin{proof}
    Since $E$ is countably rectifiable, there is an $\cH^m$-negligible set $E^0$ such that $E\setminus E^0$ is covered by countably many $m$-dimensional $C^1$ manifolds $M_h$. Writing $E_h= E\cap M_h$, we can apply \cite[Lemma A.3]{FerryLemma} to each $E_h$ to find
    \[
        \cH^\frac mp(f(E_h))=0.
    \]
    Moreover, for all $k>0$ there is a covering of $E^0$ with countably many balls $B_i(x_i,r_i)$ with $x_i\in E^0$ and of radii $r_i<1/k$ such that
    \[
        \sum_{i} r_i^m<\tfrac1k.
    \]
    Thus, setting for each $i$ the interval $I_i=(f(x_0)-Cr_i^p,f(x_0)+Cr_i^p)$, this is a covering of $f(E^0)$. As a consequence,
    \[
        \cH^{\frac m p}(f(E^0)) \le \sum_i C^\frac m p (r_i^p)^\frac m p =C^\frac mp\sum_i r_i^m< C^{\frac m p}\tfrac1k.
    \]
    It follows that $\cH^{\frac m p}(f(E^0))=0$, as we wanted.
\end{proof}
\begin{proof}[Proof of Theorem~\ref{teo:generic R4}]
    Define
    \[
        \tau(x)\coloneqq \sup\{t\in(0,T)\,:\, u(x,t)=0\}
    \]
    and denote by $\pi_x,\pi_t$ the projections on the space and time variables respectively.
    By definition of $\tau$ we have $\tau(x)=t$ if and only if $(x,t)\in\partial\{u>0\}$.
    Thus, if we denote $\cS=\pi_t(\Sigma)$ and $\cS^\infty=\pi_t(\Sigma^\infty)$ where $\Sigma^\infty$ is given by Theorem~\ref{teo:smooth covering}, we have
    \[
        \cS = \tau(\pi_x(\Sigma)),\quad \cS^\infty = \tau(\pi_x(\Sigma^\infty)),\quad \dim_\cH(\cS^\infty)=0.
    \]
    Setting $\Sigma^*=\Sigma\setminus\Sigma^\infty$, Theorem~\ref{teo:generic R4} follows if we show $\cH^1(\tau(\pi_x(\Sigma^*)))=0$.
    Given $N\in\N$ we set
    \[
        E_N \coloneqq \{x_0\in \pi_x\Sigma^* \, : \,\Sigma\cap\{|x-x_0|\le r,t\ge \tau(x_0)+Nr^2\}=\emptyset\}.
    \]
    By Proposition~\ref{prop:quadratic cleaning} we have $\pi_x(\Sigma^*)=\cup_{N\ge1}E_N$, hence it is enough to show
    \[
        \cH^1(\tau(E_N))=0\quad\forall N\ge1.
    \]
    Since $\pi_x(\Sigma_*)$ is countably $(n-2)$-rectifiable by Theorem~\ref{teo:smooth covering} and $E_N\subset \pi_X(\Sigma^*)$, we also have that $E_N$ is countably $(n-2)$-rectifiable.
    In addition, by definition we have
    \[
        \tau(x)\le \tau(y) + N|x-y|^2\quad\forall x,y\in E_N,
    \]
    which, by symmetry, yields
    \[
        |\tau(x)-\tau(y)|\le N|x-y|^2\quad \forall x,y\in E_N.
    \]
    Thus, we can apply Lemma~\ref{lem:FerryLemma} with $E=E_N$, $f=\tau$ and $m=n-2=p=2$ to find
    \[
        \cH^1(\tau(E_N))=0\quad \forall N\ge1
    \]
    as we wanted.
\end{proof}

\subsection{Quadratic cleaning at singular points}
We prove Proposition~\ref{prop:quadratic cleaning}. We will make use of the following result.
\begin{lem}[{\cite[Propositions 3.4 and 3.7]{FRS}}]\label{lem:estimates}
    Let $u\colon B_1\times[-1,1]\to[0,+\infty)$ be a bounded solution of the Stefan problem \eqref{eq:SP}. Then there is a constant $C>0$ such that
    \[
        \partial_{tt}u\ge -C\quad\text{in } B_{1/2}\times[-1/2,1/2]
    \]
    and
    \[
        (D^2u)_-\le C\partial_tu\quad \text{in } B_{1/2}\times[-1/2,1/2].
    \]
\end{lem}
Given $G_n(x,t)$ the reverse heat kernel (see \eqref{eq:gaussian kernels}) and $f\colon\R^n\times(-1,0)$ we define
\[
    H(r,f) = \int_{\R^n} f^2(\cdot,-r^2)G(\cdot,-r^2)dx.
\]
\begin{lem}\label{lem:cleaning nondegeneracy dt}
    Let $u\colon B_1\times[-1,1]\to[0,+\infty)$ be a bounded solution of the Stefan problem \eqref{eq:SP}. For all $(x_0,t_0)\in\Sigma\cap B_{1/2}\times[-1/2,1/2]$ with blow-up $p_2$ there is $c_{x_0,t_0}$ such that, denoting by $e_n$ the direction of the maximal eigenvalue of $D^2p_2$,
    \[
        \fint_{C_r\cap\{|x_n|>r/10\}} \partial_t u\ge c_{x_0,t_0} r^{-2}H(r,\zeta (u-p_2))^{1/2}\quad \forall r\in(0,1/4).
    \]
\end{lem}
\begin{proof}
     Since we need the result only for $(x_0,t_0)\in\Sigma\setminus \Sigma_{n-1}$ we will prove it only in this case. For points in $\Sigma_{n-1}$ it follows from~\cite[Proposition 8.4]{FRS}. We note, however, that the same argument used here, together with \eqref{eq:cubic blow-ups}, also applies to points in $\Sigma_{n-1}$.\\
     If by contradiction the claim is false, there is a subsequence $r_k\to0$ such that, setting $w_r = u(r\cdot,r^2\cdot)-r^2 p_2$,
     \begin{equation}\label{eq:degeneracy dt}
        \fint_{C_1\cap\{|x_n|>1/10\}} \frac{\partial_t w_{r_k}}{H(r_k,\zeta (u-p_2))^{1/2}} \to0.
     \end{equation}
     It follows from \cite[Proposition 6.7]{FRS} that, up to a subsequence,
     \[
        \frac{w_{r_k}}{H(r,\zeta (u-p_2))^{1/2}} \to q
     \]
     locally weakly in $H^1(\R^n\times(-\infty,0))$, where $q\not\equiv0$ is a quadratic caloric polynomial.
     In addition if we assume that up to a rotation $p_2 = \frac12\sum_{i=k+1}^n \mu_i x_i^2$ for some $\mu_i>0$ satisfying $\sum_{i=k+1}^n \mu_i=1$, then $q$ is of the form
     \[
        q(x,t) = At + \tfrac\nu2 \sum_{i=k+1}^n x_i^2 + \sum_{i=1}^k \tfrac{\nu_i}2x_i^2,
     \]
     where $A\ge0$, $\nu_i\le\nu$ satisfy
     \begin{equation}\label{eq:caloric condition for q}
        (n-k)\nu - A + \sum_{i=1}^k \nu_i = 0.
     \end{equation}
     However, \eqref{eq:degeneracy dt} yields $A=0$, so that $q$ is a non-zero quadratic harmonic polynomial.
     Moreover, since $\partial_{ii}p_2=0$ for all $1\le i\le k$, Lemma~\ref{lem:estimates} yields
     \[
        (\partial_{ii}w_{r_k})_-\le C \partial_t w_{r_k}.
     \]
     Dividing by $H(r_k,\zeta (u-p_2))^{1/2}$ and letting $k\to+\infty$ we find
     \[
        \partial_{ii}q=\nu_i\ge0
     \]
     for all $1\le i\le k$.
     Since $q\not\equiv0,\nu\ge\nu_i\ge0$ and $A=0$, this contradicts~\eqref{eq:caloric condition for q}.
\end{proof}

\begin{lem}[{cfr~\cite[Lemma 8.1]{FRS}}]\label{lem:cleaning dt barrier}
    Let $u\colon B_1 \times(-1,1) \to[0,\infty)$ be a solution of \eqref{eq:SP} and $(0,0)$ a singular point with first blow-up $p_2$. Assume that $e_n$ is an eigenvector for $D^2p_2$ with maximal eigenvalue, and that there exists $c>0$ such that
    \begin{equation}\label{eq:cleaning dt assumption}
        \fint_{C_r\cap\left\{|x_n|\ge \tfrac r{10}\right\}}\partial_tu \ge cr^{-2}H(r,\zeta (u-p_2))^{1/2}\quad \forall r\in(0,1).
    \end{equation}
Then there exists $C >0$ such that
\[
    \{u= 0\}\cap B_{r/2} \times[Cr^2,1) = \emptyset \quad\forall r\in(0,1).
\]
\end{lem}
\begin{proof}
    The proof is the same as \cite[Lemma 8.1]{FRS}.
    Since $\Heat(u-p_2) =-\chi_{\{u=0\}}\le0$, the function $u-p_2$ is supercaloric.
    Also, since $\partial_tu\ge0$, then $u-p_2$ is nondecreasing in time. Thus, Lemma~\ref{lem:parabolic ABP} yields
    \begin{equation}\label{eq:cleaning proof 1}
        u\ge p_2-C_1H(r,\zeta(u-p_2))^{1/2}\quad \text{in }B_r\times[-r^2/2,1).
    \end{equation}
    Since $H(r,\zeta(u-p_2))^{1/2} = o(r^2)$ and $e_n$ is an eigenvector of $D^2p_2$ with maximal eigenvalue, for any fixed $\delta>0$ small, we obtain
    \[
        \{u= 0\}\cap B_r \times[-r^2/2,1) \subset\{|x_n|\le r\delta^2\} \quad\forall r\ll1.
    \]
    Thus $\Heat(\partial_tu) = 0$ inside $B_r \cap\{|x_n|>r\delta^2\} \times[-r^2/2,1)$,
    and therefore \eqref{eq:cleaning dt assumption} and Harnack inequality imply that $\partial_tu(\cdot,-r^2/4) \ge2c_2r^{-2}H(r,\zeta(u-p_2))^{1/2}$ inside $B_r \cap\{|x_n|>r\delta\}$, for some $c_2 = c_2(n,\delta) >0$. Combining this bound with the estimate $\partial_{tt}u\ge-C$ (see Lemma~\ref{lem:estimates}), we get
    \begin{equation}\label{eq:clening proof 2}
        \partial_tu\ge c_2r^{-2}H(r,\zeta(u-p_2))^{1/2}\quad\text{in }B_r\cap\{|x_n|>r\delta\} \times[-r^2/4,c_3r]
    \end{equation}
    for some $c_3 >0$ (recall that $r^{-2}H(r,\zeta(u-p_2))^{1/2}\ge Cr$).
    In particular, combining \eqref{eq:cleaning proof 1} and \eqref{eq:clening proof 2}, we obtain for all $h\in[0,c_3r]$
    \[
        u(\cdot,-r^2/4 + h) \ge p_2-C_1H(r,\zeta(u-p_2))^{1/2} + c_2r^{-2}H(r,\zeta(u-p_2))^{1/2} h\quad\text{in }B_r \cap\{|x_n|>r\delta\}.
    \]
    Choosing $h\ge r^2/4+2C_1c_2^{-1}r^2$ and using again \eqref{eq:cleaning proof 1} we obtain
    \begin{equation}\label{eq:cleaning proof 3}
        u\ge p_2 + C_1H(r,\zeta(u-p_2))^{1/2}(-1 + 2\chi_{\{|x_n|>r\delta\}})\quad \forall(x,t) \in B_r \times[2C_1c_2^{-1}r^2,1)
    \end{equation}
    Now, let $h_\delta$ be the solution to
    \[\begin{cases}
            \Heat h_\delta = 0  &\text{in }B_1\times(0,\infty),\\
            h_\delta = 2        &\text{on }\partial B_1\cap\{|x_n>\delta\}\times[0,\infty),\\
            h_\delta=0          &\text{on }\partial B_1\cap\{|x_n|<\delta\}\times[0,\infty),\\
            h_\delta=0          &\text{at }t=0.
    \end{cases}\]
    Since $h_\delta \to2$ as $\delta\to0$, it follows that $h_\delta\ge3/2$ inside $B_{1/2}$ for all $t\ge1$, provided $\delta$ is small enough.
    Now we can observe that
    \[
        \psi(x,t) \coloneqq p_2(x) + C_1H(r,\zeta(u-p_2))^{1/2}\left(-1 + h_\delta\left(\frac xr,\frac{t-2C_1c_2^{-1}r^2}{r^2}\right)\right)
    \]
    satisfies $\Heat \psi = 1$ in $B_r\times[2C_1c^{-1}_2 r^2,\infty)$ and, by \eqref{eq:cleaning proof 3}, we have $u \ge\psi$ on the parabolic boundary $\partial_\pb B_r \times[2C_1c^{-1}_2 r^2,1)$.
    Hence, by the maximum principle, $u\ge\psi$ in $B_r$, for $t\ge2C_1c^{-1}_2r^2$. Evaluating at $t= 2C_1c^{-1}_2r^2+ r^2$ (and using that $h_\delta \ge3/2$ in $B_{1/2}$ for all $t\ge1$) we obtain
    \[
        u\ge \psi \ge p_2 + \tfrac{C_1}2H(r,\zeta(u-p_2))^{1/2} >0\quad \text{in }B_{r/2} \quad\text{for }t\ge (2C_1c_2^{-1}+1)r^2.
    \]
    and the result follows.
\end{proof}
\begin{proof}[Proof of Proposition~\ref{prop:quadratic cleaning}]
    For $(x_0,t_0)\in\Sigma\setminus\Sigma_{n-1}$ it is an immediate consequence of Lemma~\ref{lem:cleaning dt barrier} and \ref{lem:cleaning nondegeneracy dt}. For $(x_0,t_0)\in\Sigma_{n-1}$ it follows from \eqref{eq:cleaning} with $k=3$, or simply from \cite[Proposition 8.3]{FRS}.
\end{proof}
\bibliographystyle{aomalpha}
\bibliography{stefan}
\end{document}